\pdfoutput=1
\RequirePackage{ifpdf}
\ifpdf % We~are running pdfTeX in pdf mode
\documentclass[pdftex]{sigma}
\else
\documentclass{sigma}
\fi

\numberwithin{equation}{section}

\newtheorem{Theorem}{Theorem}[section]
\newtheorem*{Theorem*}{Theorem}
\newtheorem{Corollary}[Theorem]{Corollary}
\newtheorem{Lemma}[Theorem]{Lemma}
\newtheorem{Proposition}[Theorem]{Proposition}
{ \theoremstyle{definition}
	\newtheorem{Definition}[Theorem]{Definition}

	\newtheorem{Remark}[Theorem]{Remark} }

\newcommand{\R}{\mathbb{R}} % real numbers
\newcommand{\Z}{\mathbb{Z}} % integers
\newcommand{\N}{\mathbb{N}} % natural numbers
\newcommand{\C}{\mathbb{C}} % complex numbers
\newcommand{\OP}{\mathbb{OP}}
\newcommand{\restr}[1]{|_{#1}}

\newcommand{\rvline}{\hspace*{-\arraycolsep}\vline\hspace*{-\arraycolsep}}

\DeclareMathOperator{\pr}{pr}
\DeclareMathOperator{\vecspan}{span}
\DeclareMathOperator{\ad}{ad}
\DeclareMathOperator{\Id}{Id}
\DeclareMathOperator{\diag}{diag}
\DeclareMathOperator{\inv}{inv}
\DeclareMathOperator{\Span}{span}

\DeclareMathOperator{\std}{std}
\DeclareMathOperator{\GL}{GL}

\DeclareMathOperator{\SO}{SO}
\DeclareMathOperator{\Spin}{Spin}
\DeclareMathOperator{\SU}{SU}
\DeclareMathOperator{\Sp}{Sp}
\DeclareMathOperator{\U}{U}
\DeclareMathOperator{\Special}{S}

\DeclareMathOperator{\Hom}{Hom}

\DeclareMathOperator{\Ric}{Ric}

\DeclareMathOperator{\Cl}{Cl}
\DeclareMathOperator{\Mat}{M}

\DeclareMathOperator{\End}{End}
\DeclareMathOperator{\Ad}{Ad}

\usepackage{mathrsfs}

\usepackage{mathtools}

\usepackage{tikz-cd}

\usepackage{dsfont}
\usepackage{booktabs}
\usepackage{physics}
\usepackage{placeins}
\usepackage{multirow}
\usepackage{blkarray}

\usepackage{faktor}
\usepackage{upgreek}

\begin{document}
\allowdisplaybreaks

\newcommand{\arXivNumber}{2406.18337}

\renewcommand{\PaperNumber}{017}

\FirstPageHeading
	
\ShortArticleName{The Geometry of Generalised Spin$^r$ Spinors on Projective Spaces}
	
\ArticleName{The Geometry of Generalised Spin$\boldsymbol{{}^r}$ Spinors\\ on Projective Spaces}
	
\Author{Diego ARTACHO~$^{\rm a}$ and Jordan HOFMANN~$^{\rm b}$}
	
\AuthorNameForHeading{D.~Artacho and J.~Hofmann}
	
\Address{$^{\rm a)}$~Imperial College London, London SW7 2AZ, UK}
\EmailD{\href{mailto:d.artacho21@imperial.ac.uk}{d.artacho21@imperial.ac.uk}}
	
\Address{$^{\rm b)}$~King's College London, London WC2R 2LS, UK}
\EmailD{\href{mailto:jordan.2.hofmann@kcl.ac.uk}{jordan.2.hofmann@kcl.ac.uk}}
		
\ArticleDates{Received July 01, 2024, in final form March 01, 2025; Published online March 11, 2025}
		
\Abstract{In this paper, we adapt the characterisation of the spin representation via ex\-terior forms to the generalised spin$^r$ context. We find new invariant spin$^r$ spinors on the projective spaces $\mathbb{CP}^n$, $\mathbb{HP}^n$, and the Cayley plane $\mathbb{OP}^2$ for all their homogeneous realisations. Specifically, for each of these realisations, we provide a complete description of the space of invariant spin$^r$ spinors for the minimum value of $r$ for which this space is non-zero. Additionally, we demonstrate some geometric implications of the existence of special spin$^r$ spinors on these spaces.}
		
		
\Classification{53C27; 15A66; 57R15}

		\section{Introduction} \label{sec:intro}
		
		A topic of major interest in differential geometry is the existence or non-existence of \textit{special $G$-structures} on a given smooth manifold $M$; classical examples include Riemannian, complex, symplectic, and spin structures. Spin geometry, in particular, gives a way of accessing global geometric information about \textit{Riemannian spin manifolds} via sections of a certain naturally defined vector bundle called the \emph{spinor bundle}. Indeed, for a Riemannian spin manifold $M$, there are a number of results of the form:
		\begin{equation*}
			\text{$M$ carries a spinor satisfying equation $\mathscr{E}$} \implies \text{$M$ has geometric property $\mathscr{P}$.}
		\end{equation*}
		For example, it is well known that a manifold carrying a non-zero parallel spinor is Ricci-flat, and, more generally, that the existence of a non-zero real (resp.\ purely imaginary) Killing spinor implies that the metric is Einstein with positive (resp.\ negative) scalar curvature \cite{BFGK,Fri80}. Other notable examples include the bijection between generalised Killing spinors in dimension $5$ (resp.~$6$, resp.\ $7$) and hypo $\SU(2)$-structures (resp.\ half-flat $\SU(3)$-structures, resp.\ co-calibrated $\mathrm{G}_2$-structures) \cite{ACFH15,CS07}, and the spinorial description of isometric immersions into Riemannian space forms \cite{BLR,F_immersions,KS}.
		
		However, not every manifold can be endowed with a spin structure; the question then naturally arises as to whether one can apply the tools of spin geometry to non-spin manifolds. The answer is affirmative, and there are several possible approaches. The unifying idea is to consider suitable \textit{enlargements} of the spin groups, i.e., Lie groups $L_n$ equipped with homomorphisms
		\[ \Spin(n) \overset{\iota_n}{\longrightarrow} L_n \overset{p_n}{\longrightarrow} \SO(n) \]
		such that $p_n \circ \iota_n$ is the usual two-sheeted covering $\Spin(n) \to \SO(n)$. Hence, an oriented Riemannian $n$-manifold admitting a lift of the structure group to $L_n$ is a weaker condition than being spin. Following ideas by Friedrich and Trautman \cite{FT}, the so-called \textit{spinorial Lipschitz structures} have garnered much attention in recent years \cite{CS1,CS2,CS3}. These naturally arise by following the inverse approach: starting with a suitable generalisation of the concept of \textit{spinor bundle}, one investigates the enlargement $L_n$ of $\Spin(n)$ to which this bundle corresponds. These~$L_n$ are called \textit{Lipschitz groups}.
		
		Another choice of $L_n$ was introduced by Espinosa and Herrera \cite{EH16}, who had the idea of \textit{spinorially twisting} the spin group.\ In our setting, this corresponds to taking, for $r \in \N$, the groups
		\[ L_n^r = \Spin^r(n) \coloneqq ( \Spin(n) \times \Spin(r) )/\langle (-1 , -1) \rangle \]
		with the obvious homomorphisms. We say that an oriented Riemannian $n$-manifold is spin$^r$ if it admits a lift of the structure group to $\Spin^r(n)$. The case $r=1$ is the classical spin case, and the cases $r=2, 3$ give rise to spin$^{\C}$ and spin$^{\mathbb{H}}$ geometry respectively, which have been a fruitful field of study over the past decades -- see \cite{Friedrich_book, moroianu_spinc} for spin$^{\C}$ and \cite{herrera_spinq,lawson} and references therein for spin$^{\mathbb{H}}$.
		
		These structures have been characterised topologically by Albanese and Milivojevi\'{c} in \cite{AM21}, where they show that a manifold is spin$^r$ if, and only if, it can be embedded into a spin manifold with codimension $r$. In \cite{gen_spin_AL}, Lawn and the first author focused on the study of spin$^r$ structures on homogeneous spaces, establishing a bijection between $G$-invariant spin$^r$ structures on $G/H$ and certain representation-theoretical data -- see Theorem~\ref{cor:inv_gen_spin}.
		
		Analogously to the usual spin case, from a given spin$^r$ structure one can construct, for each odd $m \in \N$, the so-called \textit{$m$-twisted spin$^r$ spinor bundle} -- see Definition \ref{def:twistedbundle}. Its sections are called \textit{$m$-twisted spin$^r$ spinors} or simply \textit{spin$^r$ spinors}, and, as in the classical case, they encode geometric information: for a spin$^r$ manifold $M$, there are results of the form:
		\begin{equation*}
			\text{$M$ carries a spin$^r$ spinor satisfying equation $\mathscr{E}$} \implies \text{$M$ has property $\mathscr{P}$. }
		\end{equation*}
		Some of these results can be found in \cite{EH16,ES19}, for example (see Theorem~\ref{thm:herrera} for more explicit results):
		\begin{itemize}\itemsep=0pt
			\item The existence of a generalised Killing spin$^r$ spinor ensures a certain decomposition of the Ricci tensor \cite[Theorem~3.3]{EH16}.
			\item The existence of a parallel pure spin$^2$ \big(resp.\ spin$^3$\big) spinor implies that the manifold is K\"{a}hler (resp.\ quaternionic K\"{a}hler) \cite[Corollaries~4.10 and 4.12]{ES19}.
		\end{itemize}
		
		In this paper, we illustrate how invariant twisted spin$^r$ spinors on the projective spaces $\mathbb{CP}^n$, $\mathbb{HP}^n$ and the Cayley plane $\mathbb{OP}^2$ encode different geometric properties of these manifolds. To this end, we proceed as follows:
		\begin{enumerate}\itemsep=0pt
			\item[(1)] Consider a homogeneous realisation $M=G/H$ of the corresponding space, equipped with a generic $G$-invariant metric.
			\item[(2)] Find the minimum value of $r\in \N$ such that $M$ has a $G$-invariant spin$^r$ structure carrying non-trivial invariant $m$-twisted spin$^r$ spinors, for some odd $m \in \N$, and describe the space of such spinors.
			\item[(3)] Study the geometric properties of $M$ encoded by those invariant spin$^r$ spinors which satisfy additional algebraic properties.
		\end{enumerate}
		
		The realm of projective spaces provides a fruitful ground for study. In particular, we prove that Friedrich's construction of generalised Killing spinors on $\mathbb{CP}^3$ \cite[p.\ 146]{BFGK} cannot be generalised to higher complex dimensions, showing that this is the only dimension for which $\mathbb{CP}^{2k+1}$ has an $\Sp(k+1)$-invariant metric carrying non-trivial invariant generalised Killing spinors. We also find the spin$^{\mathbb{H}}$ spinor on $\mathbb{HP}^n$ inducing the standard quaternionic K\"{a}hler structure (see \cite[Corollary~4.12]{ES19}) with new representation-theoretical methods, and we show that this is, up to scaling, the \textit{only} $\Sp(n+1)-$invariant spin$^{\mathbb{H}}$ spinor on this space. Finally, we find that the minimum values of $r$ and $m$ such that $\mathbb{OP}^2$ carries non-trivial $F_4$-invariant $m$-twisted spin$^r$ spinors are $r=9$ and $m=3$, and the space of such spinors is four-dimensional.
		
		The computations carried out in this paper illustrate an extension of the differential forms approach to the spin representation (see, e.g., \cite{AHL}) to the context of spin$^r$ structures. These techniques allow us to express and manipulate complicated twisted spinors in an easy and readable way, finding new examples of special spin$^r$ spinors. The main contribution of this paper is, then, the fusion of the differential forms approach with the power of spin$^r$ geometry to encode geometric properties of manifolds which are not necessarily spin. Our results are summarised as follows.
		
		\begin{theorem*}
			Let $G$ be a compact, simple and simply connected Lie group acting transitively on the projective space $M = \mathbb{CP}^n, \mathbb{HP}^n$ or $\mathbb{OP}^2$. Then, the minimum values of $r,m \in \N$ $($with $m$ odd$)$ such that $M$ admits a $G$-invariant spin$^r$ structure carrying a non-zero invariant $m$-twisted spin$^r$ spinor are shown in Table~{\rm\ref{table:summary}}, together with the geometric information such spinors encode.\looseness=1
		\end{theorem*}
		
		\begin{table}\renewcommand{\arraystretch}{1.2}
			\resizebox{\textwidth}{!}{
				\begin{tabular}{c c c c c c c c}
					\toprule
					$M$ & $n$ & $G$ & $r$ & $m$ & $\dim (\Sigma_{\ast , r }^{m} )_{\mathrm{inv}}$ & Special spinors & Geometry \\ \toprule
					\multirow{3}{*}{$\mathbb{CP}^n$} & $k$ & $\SU(k+1)$ & $2$ & $1$ & $2$ & pure, parallel & K\"{a}hler--Einstein \\
					& \multirow{2}{*}{$2k+1$}	& \multirow{2}{*}{$\Sp(k+1)$} & $2$, if $k$ even & $1$ & $2$ & pure, parallel & K\"{a}hler--Einstein \\
					& & & $1$, if $k$ odd & $1$ & $2$ & generalised Killing & Einstein, nearly K\"{a}hler ($n=3$ ($\dagger$)) \\
					\midrule
					\multirow{2}{*}{$\mathbb{HP}^n$} & $2k+1$ & $\Sp(2k+2)$ & $3$ & $2k+1$ & $1$ & pure, parallel & quaternionic K\"{a}hler \\
					& $2k$ & $\Sp(2k+1)$ & -- & -- & -- & -- & --\\
					\midrule
					$\mathbb{OP}^2$ & -- & $F_4$ & $9$ & $3$ & $4$ & -- & -- \\
					\bottomrule
			\end{tabular} }
			\caption{For each compact, simple and simply connected Lie group $G$ acting transitively on $M$: the minimum values of $r$, $m$ such that $M$ admits a $G$-invariant spin$^r$ structure carrying a non-zero invariant $m$-twisted spin$^r$ spinor, and the geometric significance of these. For $\dagger$, see \cite[p.\ 146]{BFGK}. }
			\label{table:summary}
		\end{table}
		
		\FloatBarrier
		
		\section{Preliminaries}\label{section:preliminaries}
		
		We begin by introducing the necessary background definitions and results concerning spin and spin$^r$ geometry within the context of homogeneous spaces. For an introduction to spin geometry we refer the reader to \cite{BFGK,LM}, for spin$^r$ manifolds to \cite{AM21,gen_spin_AL,EH16}, and for homogeneous spaces to~\cite{Arv03}.
		
		\subsection{Invariant metrics on reductive homogeneous spaces}\label{sec:inv_metrics}
		
		Let $G/H$ be a reductive homogeneous space with reductive decomposition $\mathfrak{g} = \mathfrak{h} \oplus \mathfrak{m}$, where $\mathfrak{g}$ is the Lie algebra of $G$, $\mathfrak{h}$ is the Lie algebra of $H$ and $\mathfrak{m}$ is an $\Ad_H$-invariant complement of $\mathfrak{h}$ in~$\mathfrak{g}$. Suppose that the adjoint representation of $H$ on $\mathfrak{m}$ -- which, under the usual identifications, corresponds to the isotropy representation of the homogeneous space -- decomposes as a direct sum of irreducible components $\mathfrak{m} = \mathfrak{m}_1 \oplus \cdots \oplus \mathfrak{m}_k$. We would like to find all the $\Ad_H$-invariant inner products on $\mathfrak{m}$ (which correspond to $G$-invariant metrics on the homogeneous space $G/H$, see, e.g., \cite{Arv03}). Of course, such metrics need not exist. However, if $H$ is compact, using Weyl's trick one readily sees that they do exist. We would like to show that invariant inner products on each irreducible component are unique up to positive scaling, and that any invariant inner product on $\mathfrak{m}$ is a positive linear combination of invariant inner products on the irreducible components, yielding a $k$-parameter family of invariant metrics. This is of course false in general (consider two copies of the same irreducible representation admitting an invariant metric). However, this is essentially the only obstruction, as we shall see now.
		
		This is a very important result which appears to be well known, but it is surprisingly hard to find in the literature. We include it here with a full proof.
		
		\begin{Proposition}\label{prop:invmetrics1}
			Let $\mathfrak{g}$ be a finite-dimensional real Lie algebra and $\rho \colon \mathfrak{g} \to \End_{\mathbb{R}}(V)$ a finite-dimensional irreducible real representation of $\mathfrak{g}$. Suppose that there exists a $\rho$-invariant inner product on $V$. Then, it is unique up to positive scaling.
		\end{Proposition}
		
		\begin{proof}
			Let $B_1, B_2 \colon V \times V \to \mathbb{R}$ be two $\rho$-invariant inner products on $V$, i.e., two inner products on $V$ satisfying
			\begin{equation*}
				\forall X \in \mathfrak{g} ,\ \forall v,w \in V \colon \ B_i (\rho(X)v , w) + B_i (v,\rho(X)w) = 0 , \qquad i=1,2 .
			\end{equation*}
			Then, for each $i = 1,2$, $B_i$ defines an isomorphism of representations $\varphi_i$ between $\rho$ and its dual representation $\rho^{*}\colon \mathfrak{g} \to \End_{\mathbb{R}}(V^{*})$:
			\begin{align*}
				\varphi_i \colon \ V \to V^{*},\qquad
				v \mapsto B_i (v , -) .
			\end{align*}
			In particular, $\rho$ is self-dual. Now consider the endomorphism of representations given by $\varphi_1^{-1} \circ \varphi_2 \colon V \to V$. By Schur's lemma, the endomorphism ring of an irreducible representation (over \textit{any} ground field) is a division ring. In particular, the endomorphism ring of $\rho$ is a finite-dimensional associative division algebra over $\mathbb{R}$. By the Frobenius theorem, these are, up to isomorphism, $\mathbb{R}$, $\mathbb{C}$ and $\mathbb{H}$. So we need to consider these three cases separately.

(1) If $\End(\rho) \cong \mathbb{R}$, then there exists $\lambda \in \mathbb{R}$ such that $\varphi_1^{-1} \circ \varphi_2 = \lambda \Id_V$, which implies that~${\varphi_2 = \lambda \varphi_1}$, which in turn means that $B_2 = \lambda B_1$. As both $B_1$, $B_2$ are inner products, we must have that $\lambda > 0$. This completes the proof for the case $\End(\rho) \cong \mathbb{R}$.
				
(2) If $\End(\rho) \cong \mathbb{C}$, then there exists $J\in \End(\rho)$, with $J^2 = -\Id_V$. And any $\varphi \in \End(\rho)$ is of the form $\varphi = a \Id_V + b J$, for some $a,b\in\mathbb{R}$. In particular, $\varphi_2 = a \varphi_1 + b \varphi_1 \circ J$. But this implies that $B_2(-,-) = a B_1(-,-) + b B_1 (J-,-)$. Now we claim that, for every $v,w\in V$, $B_1(Jv,Jw)=B_1(v,w)$. Indeed, define $\tilde{B} \colon V \times V \to \mathbb{R}$ as $\tilde{B}(v,w) = B_1 (Jv,Jw)$. As~${J^2 = -\Id_V}$, $\tilde{B}$ is non-degenerate. And, as $J$ and $B_1$ are $\rho$-invariant and $B_1$ is symmetric, $\tilde{B}$ is $\rho$-invariant and symmetric. Now define $\tilde{\varphi} \colon V \to V^{*}$ by $\tilde{\varphi} (v) = \tilde{B}(v,-)$. Consider the endomorphism of representations $\varphi_{1}^{-1} \circ \tilde{\varphi}$. Then, there exist $c,d \in \mathbb{R}$ such that $\varphi_{1}^{-1} \circ \tilde{\varphi}=c \Id_V + d J$. Hence, for every ${v,w \in V}$, $\tilde{B}(v,w) = c B_1(v,w) + d B_1(Jv,w)$. As $\tilde{B}$, $B_1$ are symmetric, we have that, for every $v,w \in V$, $dB_1(Jv,w) = dB_1(v,Jw)$. Suppose $d \neq 0$. Then, for every $v,w \in V$, $B_1(Jv,w) = B_1(v,Jw)$. In particular, if~${v \neq 0}$, we would have that $B_1(Jv,Jv) = B_1\bigl(v,J^2v\bigr) = -B_1(v,v) < 0$, which contradicts positive-definiteness of $B_1$. Hence, $\tilde{B} = c B_1$, for some $c \in \mathbb{R}$. By positive-definiteness and non-degeneracy, $c > 0$. Moreover, if $v \neq 0$, $B_1(v,v)=B_1\bigl(J^2v,J^2v\bigr) = c^2 B_1(v,v)$. Hence, $c=1$. This completes the proof of the fact that, for every $v,w\in V$, $B_1(Jv,Jw) = B_1(v,w)$. Equivalently, $B_1(J-,-)$ is skew-symmetric.
				
				Now, back to our previous situation. As $B_1$ and $B_2$ are symmetric and $B_1(J-,-)$ is skew-symmetric, $b=0$. And now positive-definiteness implies that $a > 0$, finishing the proof for the case $\End(\rho) \cong \mathbb{C}$.
				
(3) Finally, suppose $\End(\rho) \cong \mathbb{H}$. Then, there exist $J_1,J_2,J_3\in \End(\rho)$ linearly independent, with $J_i^2 = -\Id_V$ for $i=1,2,3$ and $J_1 J_2 J_3 =-\Id_V$, such that any element $\varphi \in \End(\rho)$ is of the form $a \Id_V + b_1 J_1 + b_2 J_2 + b_3 J_3$, for some $a,b_1,b_2,b_3 \in \mathbb{R}$. Hence, in particular, $\varphi_1^{-1} \circ \varphi_2$ is of this form, and thus $B_2(-,-) = a B_1(-,-) + b_1 B_1 (J_1-,-) + b_2 B_1 (J_2-,-) + b_3 B_1 (J_3-,-)$, for some $a,b_1,b_2,b_3 \in \mathbb{R}$. As in the previous case, one shows that, for every~${i \in \{1,2,3\}}$, $B_1(J_i-,-)$ is skew-symmetric. For completeness, let us show it for $J_1$. Define $\tilde{B} \colon V \times V \to \mathbb{R}$ as $\tilde{B}(v,w) = B_1 (J_1v,J_1w)$. As $J_1^2 = -\Id_V$, $\tilde{B}$ is non-degenerate. And, as $J_1$ and $B_1$ are $\rho$-invariant and $B_1$ is symmetric, $\tilde{B}$ is $\rho$-invariant and symmetric. Now define $\tilde{\varphi} \colon V \to V^{*}$ by $\tilde{\varphi} (v) = \tilde{B}(v,-)$. Consider the endomorphism of representations~${\varphi_{1}^{-1} \circ \tilde{\varphi}}$. Then, there exist $c,d_1,d_2,d_3 \in \mathbb{R}$ such that $\varphi_{1}^{-1} \circ \tilde{\varphi}=c \Id_V + d_1 J_1 + d_2J_2 +d_3J_3$. Hence, for every $v,w \in V$, $\tilde{B}(v,w) = c B_1(v,w) + d_1 B_1(J_1v,w) + d_2 B_1(J_2v,w) + d_3 B_1(J_3v,w)$. As $\tilde{B}$ and $B_1$ are symmetric, $d_1 B_1(J_1-,-) + d_2 B_1(J_2-,-) + d_3 B_1(J_3-,-)$ is symmetric. Define $Q= d_1J_1 + d_2J_2 +d_3J_3$, so that $B_1(Q-,-)$ is symmetric. Note that $Q^2 = -\bigl(d_1^2+d_2^2+d_3^2\bigr)\Id_V$. Now, for $0 \neq v \in V$, $B_1(Qv,Qv) = B_1\bigl(v,Q^2v\bigr)=-\bigl(d_1^2+d_2^2+d_3^2\bigr) B_1(v,v) \leq 0$. By positive-definiteness of $B_1$, we must have $d_1^2+d_2^2+d_3^2=0$, and hence \smash{$\tilde{B}(v,w) = c B_1(v,w)$}. Again by positive definiteness of~$B_1$, $c>0$. And, for $0 \neq v \in V$,
\[
B_1(v,v)=B_1\bigl(J_1^2 v, J_1^2v\bigr) = \tilde{B}(J_1 v, J_1 v) = c B_1(J_1v,J_1v) = c \tilde{B}(v,v) = c^2 B_1(v,v),
\]
 so $c=1$. This shows that $B_1(J_1-,J_1-)= B_1(-,-)$ or, equivalently, that $B_1(J_1-,-)$ is skew-symmetric. One can repeat this reasoning with $J_2$ and $J_3$ to show that $B_1(J_2-,-)$ and $B_1(J_3-,-)$ are skew-symmetric. Going back to our previous situation, we have that $b_1 B_1(J_1-,-)\allowbreak + b_2 B_1(J_2-,-) + b_3 B_1(J_3-,-)$ is skew-symmetric. And, as $B_1$, $B_2$ are symmetric, this bilinear form is also symmetric. Hence, it is $0$. Therefore, by positive-definiteness of $B_1$, $b_1 J_1 + b_2J_2 +b_3J_3=0$, showing that $b_1=b_2=b_3=0$, and hence that~${B_2=a B_1}$. And, by positive-definiteness, $a>0$.

This concludes the proof.
		\end{proof}
		
		\begin{Proposition}\label{prop:invmetrics2}
			Let $\mathfrak{g}$ be a finite-dimensional Lie algebra and $(\rho_1,V_1)$, $(\rho_2,V_2)$ two \textit{non-isomorphic self-dual} irreducible real representations of $\mathfrak{g}$. Then, $V_1 \perp V_2$ with respect to any $(\rho_1 \oplus \rho_2)$-invariant inner product on $V_1 \oplus V_2$.
		\end{Proposition}
		
		\begin{proof}
			Let $B$ be a $(\rho_1 \oplus \rho_2)$-invariant inner product on $V_1 \oplus V_2$. If $V_1$ is not $B$-orthogonal to~$V_2$, then we get a non-zero morphism of representations $V_1 \to V_2^{*}$, namely $v \mapsto B(v,-)\restr{V_2}$. By Schur's lemma, this would be an isomorphism, which is a contradiction since $V_2$ is self-dual and $V_1$ is not isomorphic to $V_2$.
		\end{proof}

		Finally, we have the result we were after.

		\begin{Theorem}\label{thm:inv_metrics}
			Let $G/H$ be a reductive homogeneous space with reductive decomposition $\mathfrak{g} = \mathfrak{h} \oplus \mathfrak{m}$, where $\mathfrak{g}$ is the Lie algebra of $G$, $\mathfrak{h}$ is the Lie algebra of $H$ and $\mathfrak{m}$ is an $\Ad_H$-invariant complement of $\mathfrak{h}$ in $\mathfrak{g}$. Suppose that the adjoint $($isotropy$)$ representation of $H$ on $\mathfrak{m}$ decomposes as a direct sum of \textit{pairwise non-isomorphic} irreducible components $\mathfrak{m} = \mathfrak{m}_1 \oplus \cdots \oplus \mathfrak{m}_k$. Suppose that $\mathfrak{m}$ admits an $\Ad_H$-invariant inner product. Then, $\Ad_H$-invariant inner products on each irreducible component exist and are unique up to positive scaling, and any $\Ad_H$-invariant inner product on $\mathfrak{m}$ is a positive linear combination of $\Ad_H$-invariant inner products on the irreducible components.
		\end{Theorem}
		\begin{proof}
			An $\Ad_H$-invariant inner product on $\mathfrak{m}$ restricts to an $\Ad_H$-invariant inner product on each irreducible component $\mathfrak{m}_i$. Hence, $\mathfrak{m}_i$ is self-dual. Now, apply Propositions \ref{prop:invmetrics1} and \ref{prop:invmetrics2}.
		\end{proof}

		\subsection{Some notation} \label{sec:notation}
		
		\begin{enumerate}\itemsep=0pt
			\item[(1)] If $V$ is a representation of a Lie group $G$, and $H\subseteq G$ is a subgroup, then we shall denote the restricted representation by $V\rvert_H$. We shall use analogous notation for restrictions of Lie algebra representations to Lie subalgebras.
			\item[(2)] Throughout the computations carried out in this paper, we will repeatedly use some matrix notation and identities, taken from \cite[p.\ 9]{AHL}. We will denote by \smash{$E^{(n)}_{i,j}$} (resp.\ \smash{$F^{(n)}_{i,j}$}) the elementary $n \times n$ skew-symmetric (resp.\ symmetric) matrix given by
			\begin{align*}
				E_{i,j}^{(n)}=\begin{blockarray}{r*{4}{ >{}c}}
					& & i & j & \\
					\begin{block}{ r!{\,}[cccc]}
						& & & \vdots & \\
						i & & & -1 & \hdots \\
						j &\hdots & 1& & \\
						& & \vdots & & \\
					\end{block}
				\end{blockarray}
				, \qquad F_{i,j}^{(n)} = \begin{blockarray}{r*{4}{ >{}c}}
					& & i & j & \\
					\begin{block}{ r!{\,}[cccc]}
						& & & \vdots & \\
						i & & & 1 & \hdots \\
						j &\hdots & 1& & \\
						& & \vdots & & \\
					\end{block}
				\end{blockarray} .
			\end{align*}
			By convention, the matrix \smash{$F^{(n)}_{i,i}$} has all the entries equal to zero except for the $(i,i)$ entry, which is $1$. We will denote by $B_0$ the bilinear form on the space of matrices of appropriate size given by
			\begin{equation*}
				B_0(X,Y) \coloneqq - \Re ( \tr (XY) ) ,
			\end{equation*}
			where $\Re(z)$ denotes the real part of $z$ and $\tr(A)$ is the trace of the matrix $A$. Finally, if~$\{e_i\}_i$ is an orthonormal basis for some vector space $V$ with respect to an inner product~$B$, we shall denote by $e_{i,j}:= e_i\wedge e_j$ the standard basis elements for $\mathfrak{so}(V,B) \cong \Lambda^2 V$, sending~${e_i \mapsto e_j}$ and $e_j \mapsto -e_i$.
		\end{enumerate}
		
\subsection[Invariant spin\^{}r structures]{Invariant spin$^{\boldsymbol{r}}$ structures}
		
		Denote by $\SO(n)$ the special orthogonal group, and let $\lambda_n \colon \Spin(n) \to \SO(n)$ be the standard two-sheeted covering. This map induces an isomorphism at the level of Lie algebras, and its inverse $\rho \colon \mathfrak{so}(n) \cong \Lambda^2 \R^n \to \mathfrak{spin}(n) \subseteq \mathbb{C}\mathrm{l}(n)$ is given by $2 e_i \wedge e_j \mapsto e_i \cdot e_j$. If $f$ is any map with codomain $\mathfrak{so}(n)$, we will refer to \smash{$\widetilde{f} \coloneqq \rho \circ f$} as the \textit{spin lift of $f$}.
		
		For $r \in \mathbb{N}$, we define the group
		\[
		\Spin^r(n) \coloneqq (\Spin(n) \times \Spin(r) ) / \mathbb{Z}_2 ,
		\]
		where $\mathbb{Z}_2 = \langle(-1,-1)\rangle \subseteq \Spin(n) \times \Spin(r)$. Note that $\Spin^1(n) = \Spin(n)$, $\Spin^2(n) = \Spin^{\mathbb{C}}(n)$ and $\Spin^3(n) = \Spin^{\mathbb{H}}(n)$, and that there are natural homomorphisms
		\begin{alignat*}{3}
			& \lambda^{r}_n \colon\ \Spin^{r}(n) \to \SO(n) , \qquad&& [\mu,\nu] \mapsto \lambda_n(\mu),& \\
			& \xi^{r}_n \colon\ \Spin^{r}(n) \to \SO(r), \qquad&& [\mu,\nu] \mapsto \lambda_r(\nu).&
		\end{alignat*}
		We recall a topological result that we will use multiple times throughout the text.

		\begin{Proposition}[\cite{gen_spin_AL}]\label{prop:fg}
			The map $\varphi^{r,n} \colon \Spin^r(n) \to \SO(n) \times \SO(r)$ defined by $\lambda^{r}_n \times \xi^{r}_n$ is a~two-sheeted covering. Moreover,
			\begin{enumerate}\itemsep=0pt
				\item[$(1)$] $\varphi^{2,2}_{\sharp}\bigl(\pi_1\bigl( \Spin^2(2)\bigr) \bigr) = \langle(1,\pm 1)\rangle \subseteq \mathbb{Z} \times \mathbb{Z} \cong \pi_1 ( \SO(2) \times \SO(2))$,
				\item[$(2)$] for $n \geq 3$, $\varphi^{2,n}_{\sharp} \bigl( \pi_1 \bigl(\Spin^2(n)\bigr)\bigr) = \langle(1,1)\rangle \subseteq \mathbb{Z}_2 \times \mathbb{Z} \cong \pi_1 ( \SO(n) \times \SO(2))$,
				\item[$(3)$] for $r,n \geq 3$, $\varphi^{r,n}_{\sharp}(\pi_1 (\Spin^r(n) ) ) = \langle (1,1) \rangle \subseteq \mathbb{Z}_2 \times \mathbb{Z}_2 \cong \pi_1 ( \SO(n) \times \SO(r))$,
			\end{enumerate}
			where we always take the identifications $\pi_1 (\SO(n) \times \SO(r)) \cong \pi_1(\SO(n)) \times \pi_1(\SO(r))$.
		\end{Proposition}
		These \textit{enlargements} or \textit{twistings} of the spin group give rise to the main players in this paper.
		\begin{Definition}
			Let $M$ be an oriented Riemannian $n$-manifold with principal $\SO(n)$-bundle of positively oriented orthonormal frames $FM$. A \textit{spin$^r$ structure} on $M$ is a reduction of the structure group of $FM$ along the homomorphism $\lambda^r_n$. In other words, it is a pair $(P,\Phi)$ consisting~of
			\begin{itemize}\itemsep=0pt
				\item a principal $\Spin^r(n)$-bundle $P$ over $M$, and
				\item a $\Spin^r(n)$-equivariant bundle homomorphism $\Phi \colon P \to FM$, where $\Spin^r(n)$ acts on $FM$ via $\lambda^r_n$.
			\end{itemize}
			If there is no risk of confusion and $\Phi$ is clear from the context, we shall simply denote such a structure by $P$. The principal $\SO(r)$-bundle associated to $P$ along $\xi^r_n$ is called the \textit{auxiliary bundle} of the spin$^r$ structure, and it is denoted by $\hat{P}$.
			
			If $(P_1,\Phi_1)$ and $(P_2,\Phi_2)$ are spin$^r$ structures on $M$, an \textit{equivalence} of spin$^r$ structures from $(P_1,\Phi_1)$ to $(P_2,\Phi_2)$ is a $\Spin^r(n)$-equivariant diffeomorphism $f \colon P_1 \to P_2$ such that $\Phi_1 = \Phi_2 \circ f$.
			
			If, moreover, $M = G/H$ is a Riemannian homogeneous space, we say that a spin$^r$ structure~${(P,\Phi)}$ on $M$ is \textit{$G$-invariant} if $G$ acts smoothly on $P$ by $\Spin^r(n)$-bundle homomorphisms and $\Phi$ is $G$-equivariant.
		\end{Definition}
		
		$G$-invariant spin$^r$ structures on $G/H$ are in one-to-one correspondence with representation-theoretical data.
		\begin{Theorem}[\cite{gen_spin_AL}]\label{cor:inv_gen_spin}
			Let $G/H$ be an $n$-dimensional oriented Riemannian homogeneous space with $H$ connected and isotropy representation $\sigma \colon H \to \SO(n)$. Then, there is a bijective correspondence between
			\begin{itemize}\itemsep=0pt
				\item $G$-invariant spin$^r$ structures on $G/H$ modulo $G$-equivariant equivalence of spin$^r$ structures, and
				\item Lie group homomorphisms $\varphi \colon H \to \SO(r)$ such that $\sigma \times \varphi \colon H \to \SO(n) \times \SO(r)$ lifts to a~homomorphism $\phi \colon H \to \Spin^r(n)$ along $\lambda_n^r$, modulo conjugation by an element of $\SO(r)$.
			\end{itemize}
			Explicitly, to such a $\varphi$ corresponds the spin$^r$ structure $(P,\Phi)$ with $P = G \times_{\phi} \Spin^r(n)$ and $ \Phi \colon P \to FM \cong G \times_{\sigma} \SO(n)$ given by $[g,x] \mapsto [g,\lambda_{n}^{r}(x)]$.
		\end{Theorem}
		\begin{Definition} \label{deff:types}
			For an oriented Riemannian homogeneous space $M=G/H$, its \textit{$G$-invariant spin type} $\Sigma(M,G)$ is defined by
			\[
			\Sigma(M,G):= \text{min}\{ r\in \mathbb{N} \mid M \text{ admits a $G$-invariant spin$^r$ structure}\} .
			\]
		\end{Definition}
		
		\subsection{Exterior forms approach to the spin representation} \label{sec:diff_forms_approach}
		
		It is well known -- see, e.g., \cite{LM} -- that, for $n \in \N$, the complexification $\mathbb{C}\mathrm{l}(n)$ of the real Clifford algebra $\Cl(n)$ satisfies
		\begin{equation}\label{eq:clifford_iso}
			\mathbb{C}\mathrm{l}(n) \cong \begin{cases*} \Mat_{2^k} (\mathbb{C}) & if $n=2k$, \\
				\Mat_{2^k} (\mathbb{C}) \oplus \Mat_{2^k} (\mathbb{C}) & if $n=2k+1$.
			\end{cases*}
		\end{equation}
		For $n = 2k$ or $2k+1$, define
		\[ \Sigma_n = \bigl(\mathbb{C}^2\bigr)^{\otimes k} , \]
		and let $s_k \colon \Mat_{2^k} (\mathbb{C}) \to \End_{\C}(\Sigma_n)$ be the standard representation of $\Mat_{2^k} (\mathbb{C})$. The \textit{spin representation} $\Delta_n \colon \mathbb{C}\mathrm{l}(n) \to \End_{\mathbb{C}}(\Sigma_n)$ is defined by
		\begin{equation*}
			\Delta_n \coloneqq \begin{cases*} s_k & if $n=2k$, \\
				s_k \circ \pr_j & if $n=2k+1$,
			\end{cases*}
		\end{equation*}
		where $\pr_j$ is the projection onto the $j$-th factor. Note that, for odd $n$, there are two non-isomorphic irreducible representations of $\C \mathrm{l} (n)$, and we are \textit{choosing} one of them (we shall specify \textit{which} one below). We will also denote by $\Delta_n$ its restriction to $\Spin(n)$ when there is no risk of confusion. This restriction is independent of the choice of representation for odd $n$.
		
		It is useful to have an explicit description of the spin representation which does not use the isomorphism \eqref{eq:clifford_iso}. We will describe one here, which we refer to as the \textit{exterior forms approach} to the spin representation. Similar realizations have appeared, e.g., in \cite{GW09,LM} and in the early supergravity literature (see, e.g., \cite{Gillard_2005}). More details and examples, using our precise conventions, can be found, e.g., in \cite{AHL,jordansasakian}.
		
		Suppose $n = 2k+1$, and let $(e_0,\dots,e_{2k})$ be the standard basis of $\mathbb{R}^n$. Its complexification decomposes as
$
			\mathbb{C}_0 \oplus L \oplus L' $,
		where $\mathbb{C}_0 = \vecspan_{\mathbb{C}}\{u_0 \coloneqq {\rm i} e_0\}$ and
		\begin{equation*}
			L \coloneqq \vecspan_{\mathbb{C}} \left\{ x_j \coloneqq \frac{1}{\sqrt{2}} ( e_{2j-1} - {\rm i} e_{2j} ) \right\}_{j=1}^{k} , \qquad L' \coloneqq \vecspan_{\mathbb{C}} \left\{ y_j \coloneqq \frac{1}{\sqrt{2}} ( e_{2j-1} + {\rm i} e_{2j} ) \right\}_{j=1}^{k} .
		\end{equation*}
		Note that $\dim_{\mathbb{C}}(\Lambda^{\bullet}L') = 2^k$, and $\mathbb{C}\mathrm{l}(n)$ acts on $\Lambda^{\bullet}L'$ by extending
		\[
			x_j \cdot \eta \coloneqq {\rm i}\sqrt{2} x_j \lrcorner \eta , \qquad y_j \cdot \eta \coloneqq {\rm i}\sqrt{2} y_j \wedge \eta , \qquad u_0 \cdot \eta \coloneqq -\eta_{\mathrm{even}} + \eta_{\mathrm{odd}} ,
		\]
		where $\eta_{\mathrm{even}}$ and $\eta_{\mathrm{odd}}$ are, respectively, the even and odd parts of $\eta \in \Lambda^{\bullet} L'$. Hence,
		\begin{equation}\label{eq:def_action}
			e_{2j-1} \cdot \eta \coloneqq {\rm i} ( x_j \lrcorner \eta + y_j \wedge \eta ) ,\qquad e_{2j} \cdot \eta \coloneqq y_j \wedge \eta - x_j \lrcorner \eta , \qquad e_0 \cdot \eta \coloneqq {\rm i} \eta_{\mathrm{even}} -i \eta_{\mathrm{odd}} .
		\end{equation}
		This representation is isomorphic to $\Delta_n$ for $n=2k+1$ (the other possible choice of irreducible representation of $\mathbb{C}\mathrm{l}(n)$ corresponds to letting $e_0$ act by the negative of what is established in~\eqref{eq:def_action}). To obtain it for $n=2k$, repeat all the above ignoring everything with a zero subscript. These representations have an invariant Hermitian product, which we shall denote by $\langle \cdot , \cdot \rangle$, and an associated norm $\abs{\cdot}$, for which the basis
		\[
 \{ y_{j_1, \dots, j_k} \coloneqq y_{j_1} \wedge \cdots \wedge y_{j_k} \mid 0 \leq k \leq n , \, 1 \leq j_1 < \cdots < j_k \leq n \}
		\]
		is orthonormal.\footnote{Here we use the convention that the empty wedge product (i.e., the case $k=0$) is equal to $1$.}
		
		\subsection[Invariant spin\^{}r spinors]{Invariant spin$^{\boldsymbol{r}}$ spinors}
		
		A classical spin structure allows us to build a spinor bundle. Similarly, a spin$^r$ structure naturally induces a family of complex vector bundles as follows:
		\begin{Definition}\label{def:twistedbundle}
			Let $M$ be an $n$-dimensional oriented Riemannian manifold admitting a spin$^r$ structure $(P,\Phi)$. For $m \in \N$ odd, its \textit{$m$-twisted spin$^r$ spinor bundle} is defined by
			\[ \Sigma_{n,r}^m M \coloneqq P \times_{\Delta_{n,r}^{m}} \Sigma_{n,r}^{m} , \]
			with the natural projection to $M$ induced by that of $P$, where
			\[ \Delta_{n,r}^{m} \coloneqq \Delta_n \otimes \Delta_{r}^{\otimes m} , \qquad \Sigma_{n,r}^{m} \coloneqq \Sigma_n \otimes \Sigma_{r}^{\otimes m} , \]
			where $\Delta_{n,r}^{m}$ is viewed as a representation of $\Spin^r(n)$.
		\end{Definition}
		The requirement that $m$ be odd in Definition \ref{def:twistedbundle} comes from the fact that $\Delta_{n,r}^{m}$, which is a~representation of $\Spin(n) \times \Spin(r)$, descends to a representation of $\Spin^r(n)$ if and only if $m$ is odd.
		
		If, moreover, $M = G/H$ is a homogeneous space and $(P,\Phi)$ is $G$-invariant, then there exists a homomorphism $\varphi \colon H \to \SO(r)$ such that $\sigma \times \varphi$ lifts to a map $\phi \colon H \to \Spin^r(n)$ (see Theorem~\ref{cor:inv_gen_spin}), and this bundle takes the form
		\[ \Sigma_{n,r}^m M = G \times_{\Delta_{n,r}^{m} \circ \phi} \Sigma_{n,r}^{m} . \]
		Sections of $\Sigma_{n,r}^m M$ are called \textit{$m$-twisted spin$^r$ spinors} -- if there is no risk of confusion, we will just refer to them as \textit{spin$^r$ spinors} or simply \textit{spinors}. They are identified with $H$-equivariant maps $\psi \colon G \to \Sigma_{n,r}^{m}$, and $G$ acts on the space of spinors by
$
		(g \cdot \psi ) (g') \coloneqq \psi \bigl( g^{-1} g' \bigr)$.
		$G$-invariant spinors correspond, then, to \textit{constant} $H$-equivariant maps $G \to \Sigma_{n,r}^{m}$, which in turn correspond to elements of $\Sigma_{n,r}^{m}$ which are stabilised by $H$. If $H$ is connected, these are just elements of $\Sigma_{n,r}^{m}$ which are annihilated by the differential action of the Lie algebra of $H$. We denote the space of invariant $m$-twisted spin$^r$ spinors by \smash{$( \Sigma_{n,r}^{m} )_{\mathrm{inv}}$}.
		
		\begin{Remark}
			It should be noted that the exterior forms approach to the spin representation described above does \emph{not} in general give an identification of (classical) spinors with globally defined differential forms; the isotropy action on $\Sigma_n$ is not generally equivalent (as representations) to the isotropy action on $\Lambda^{0,\bullet}\mathfrak{m}$. Rather, this realisation of spinors via exterior forms is purely algebraic, and is often non-canonical (i.e., it depends on the choice of basis for $\mathfrak{m}$). Notable exceptions exist in the presence of certain special geometric structures -- see, e.g., \cite[Remark 3.9]{AHL} and the proof of \cite[Theorem 5.10]{jordansasakian}. On the other hand, there are other various constructions (so-called \emph{squaring constructions}) which associate (real) differential forms to spinors
			(see, e.g., \cite{CLS21,LM,Wang}). One common such construction is to associate to a spinor $\psi$ the $k$-form
			\[ \omega_{(k)}(X_1,\dots,X_k):= \Re \langle (X_1\wedge \dots X_k)\cdot \psi, \psi \rangle \qquad \text{for all}\quad X_1,\dots, X_k \in TM, \]
			and it is well known that if $\psi$ is a Killing spinor then $\omega_{(1)}$ (or, more precisely, its dual vector field) is a Killing vector field (see, e.g., \cite[Section 1.5]{BFGK}). It should be noted that one obtains quite often $\omega_{(k)}=0$, even for non-vanishing spinors $\psi$ (see, e.g., \cite[Table 6]{AHL}). In particular, the differential forms associated to an invariant spinor in this manner do not seem to be heavily influenced by its realisation in the exterior form model of the spin representation (which itself may be non-canonical).
		\end{Remark}
		
In a similar spirit to Definition \ref{deff:types}, we make the following definition.
		\begin{Definition}

			For an oriented Riemannian homogeneous space $M=G/H$, the \textit{$G$-invariant spinor type} of $M$ is defined by
			\[
			\sigma(M,G) \coloneqq \min\left\{ r \in \mathbb{N} \, \bigg\rvert \,
			\begin{aligned}
				&M \text{ admits a } G\text{-invariant spin}^r \text{ structure} \\
				&\text{with } (\Sigma_{n,r}^m)_{\mathrm{inv}} \neq 0 \text{ for some odd } m
			\end{aligned}
			\right\} .
			\]
		\end{Definition}
		\begin{Remark}
			The $G$-invariant spinor type $\sigma(M^n,G)$ is well defined, and it satisfies
			$
			1 \leq \sigma(M,G) \leq n$.
			This is because the $G$-invariant $\Spin^n$ structure on $M$ determined by taking~${\varphi \colon H \to \SO(n)}$ to be equal to the isotropy representation always carries a non-zero invariant $1$-twisted spin$^n$ spinor -- see \cite[Proposition 3.3]{EH16}.
		\end{Remark}
		
		The requirement that $r$ be minimal in the definition of the $G$-invariant spinor type is motivated by the next proposition, which shows that passing from a spin$^r$ structure to any spin$^{r'}$ structure ($r'>r$) induced by it via the obvious inclusion $\Spin^r(n) \hookrightarrow \Spin^{r'}(n)$ leads to redundancies. Before stating the proposition, we introduce some terminology which will be useful in describing the relationship between the structures:
		\begin{Definition}
			Let $M^n=G/H$ be an oriented Riemannian homogeneous space. We say that a spin$^r$ structures $P_r$ and a spin$^{r'}$ structure $P_{r'}$ ($r\leq r'$) on $M$ are in the same \emph{lineage} if~${P_{r'} \cong P_r \times_{\iota} \Spin^{r'}(n)}$, where $\iota\colon \Spin^r(n) \hookrightarrow \Spin^{r'}(n)$ is the natural inclusion map induced by the inclusion $\SO(r) \hookrightarrow \SO(r')$ as the lower right-hand $r\times r$ block.
		\end{Definition}
		
		\begin{Proposition}Let $M^n=G/H$ be an oriented Riemannian homogeneous space with connected isotropy group $H$, equipped with a $G$-invariant spin$^r$ structure $P_r$. Furthermore, for any~${r'\geq r}$ consider the invariant spin$^{r'}$ structure $P_{r'}$ in the lineage of $P_r$. If \smash{$\psi \in (\Sigma^m_{n,r})_{\inv}$} is an invariant $m$-twisted spin$^r$ spinor, then it induces an invariant $m'$-twisted spin$^{r'}$ spinor for any~${m'\geq m}$ $(m$, $m'$ odd$)$, i.e., there is an inclusion
			\[
			(\Sigma^m_{n,r})_{\mathrm{inv}} \hookrightarrow \bigl(\Sigma^{m'}_{n,r'}\bigr)_{\mathrm{inv}} \qquad \text{for all}\quad r'\geq r, \quad m'\geq m \quad\text{ $(m$, $m'$ odd$)$}.
			\]
		\end{Proposition}
		\begin{proof}
			It suffices to prove the result for $r'\in \{r, r+1\}$. Suppose first that $r'=r+1$, and let~${\varphi \colon H\to \SO(r)}$ be an auxiliary homomorphism corresponding to $P_r$ in the sense of Theorem~\ref{cor:inv_gen_spin}. Denoting by $\sigma\colon H\to \SO(n)$ the isotropy representation, we begin by observing that the invariant spin$^{r+1}$ structure in the lineage of $P_r$ is induced by the lift of the homomorphism~${\sigma\times \varphi' \colon H \to \SO(n)\times \SO(r+1)}$ given by the composition of $\sigma \times \varphi \colon H \to \SO(n)\times \SO(r)$ with the inclusion $\SO(n)\times \SO(r)\hookrightarrow \SO(n)\times \SO(r+1)$. In particular, $\mathfrak{h}$ acts on ${\Sigma_{n,r+1}^{m'}= \Sigma_n \otimes \Sigma_{r+1}^{\otimes m'}}$ by the (tensor product action associated to the) lift of $\sigma_*$ on the~$\Sigma_n$ factor and the lift of $\varphi_*$ on the~$\Sigma_{r+1}$ factors. We now split into two cases based on the parity of $r$. Supposing first that $r$ is even, we have \smash{$\Sigma_{r+1}\rvert_{\mathfrak{spin}(r)^{\C}} \simeq \Sigma_{r}$} as $\mathfrak{spin}(r)^{\C}$ representations, and therefore $\Sigma^{m'}_{n,r+1}\rvert_{\mathfrak{h}^{\C}}\simeq \Sigma^{m'}_{n,r} $ as $\mathfrak{h}^{\C}$-representations. Since $m$, $m'$ are both odd we have $m'-m = 2k$ for some $k\geq 0$, and therefore
 \[
 \Sigma_r^{\otimes m'}\rvert_{\mathfrak{spin}(r)^{\C}}\simeq \Sigma_r^{\otimes m} \otimes \underbrace{\Sigma_r \otimes \dots \otimes \Sigma_r }_{2k \text{ copies}}.
\]
 But $\Sigma_r$ is a self-dual representation of $\mathfrak{spin}(r)^{\C}$, hence also a self-dual representation of $\mathfrak{h}^{\C}$, so~$\Sigma_r^{\otimes 2k}$ contains a copy of the trivial $\mathfrak{h}$-representation. The corresponding $H$-representation thus also contains a trivial subrepresentation since $H$ is connected. In particular, there is a~copy of $\Sigma^m_{n,r}\rvert_H$ inside $\Sigma^{m'}_{n,r}\rvert_H$ and the result in this case follows. Suppose now that $r$ is odd, and denote by \smash{$\Sigma_{r+1} = \Sigma^+_{r+1} \oplus \Sigma^-_{r+1}$} the splitting into positive and negative half-spinor spaces. Then we have~${\Sigma_{r+1}^+\rvert_{\mathfrak{spin}(r)^{\C}} \simeq \Sigma_r}$, hence $\Sigma^{m'}_{n,r+1}\rvert_{\mathfrak{h}^{\C}}$ contains a copy of $\Sigma_{n,r}^{m'}$, and the result in this case then follows by the same argument as in the even case. We have shown the result holds for~${r'=r+1}$ (hence for all $r'>r$), and all that remains is to consider the case $r'=r$. The result in this case follows by arguing exactly as above, using the fact that $\Sigma_r$ is a self-dual $\mathfrak{spin}(r)^{\C}$ representation to find a copy of the trivial representation in \smash{$\Sigma_r^{\otimes (m'-m)}$}.
		\end{proof}

		\subsection[Special spin\^{}r spinors]{Special spin$^{\boldsymbol{r}}$ spinors}
		
		In the classical spin setting, it is well known that spinors satisfying certain additional properties carry geometric information about the manifold. Some of the most widely studied examples are the so-called \emph{Riemannian Killing spinors}, which are solutions of the differential equation~${
\nabla^g_X \psi = \lambda X\cdot \psi}$ for all $ X\in TM$
		(here $\nabla^g$ denotes the spinorial connection induced by the Levi-Civita connection, and $\lambda \in \R$). We refer the reader to \cite{Bar93,BFGK,FKMS97,Gru90}, among others, for a detailed exposition of their basic properties and relationship to geometric structures in low dimensions. Another class of important special spinors are the \emph{pure spinors}, which are defined by the algebraic condition that their annihilator inside $T^{\C}M$ (with respect to Clifford multiplication) is a maximal isotropic subbundle. Such spinors correspond, uniquely up to scaling, with orthogonal almost complex structures on the manifold~-- see \cite[Chapter 9]{LM} for details.
		
As in the classical spin case, special spin$^r$ spinors also encode geometric properties. In analogy with pure spinors, we define.

		\begin{Definition}[\cite{ES19}]
			Let $\psi \in \Sigma_{n,r}^m$, $X,Y \in \mathbb{R}^n$ and $1 \leq k<l \leq r$, and let $(\hat{e}_1, \dots, \hat{e}_r)$ be the standard basis of $\mathbb{R}^r$. The real $2$-form \smash{$\eta^{\psi}_{kl}$} and the skew-symmetric endomorphism \smash{$\hat{\eta}_{kl}^{\psi}$} associated to $\psi$ are defined by
			\[
			\eta_{kl}^{\psi} (X,Y) \coloneqq \Re \langle (X \wedge Y ) \cdot (\hat{e}_k \cdot \hat{e}_l) \cdot \psi , \psi \rangle , \qquad
			\hat{\eta}_{kl}^{\psi} (X) \coloneqq \bigl( \eta_{kl}^{\psi} (X, \cdot) \bigr)^{\sharp} ,
			\]
			where $X \wedge Y = X \cdot Y + \langle X,Y \rangle \in \mathfrak{spin}(n)$ and $\hat{e}_k \cdot \hat{e}_l \in \mathfrak{spin}(r)$. We say that $\psi$ is \textit{pure} if
			\[ \bigl( \hat{\eta}_{kl}^{\psi} \bigr)^2 = -\Id_{\mathbb{R}^n} \qquad \text{and} \qquad \bigl( \eta_{kl}^{\psi} + 2 \hat{e}_k \cdot \hat{e}_l \bigr) \cdot \psi = 0 \qquad (\text{only for}\ r \geq 3) ,
			\]
			for all $1 \leq k < l \leq r$. An $m$-twisted spin$^r$ spinor on a manifold is pure if it is pure at every point.
		\end{Definition}
		
		It is clear that an invariant spin$^r$ spinor on a homogeneous space is pure if, and only if, it is pure at one point.
		
		We are also interested in various differential equations that a spin$^r$ spinor might satisfy. Recall that the Levi-Civita connection on a spin manifold naturally induces a connection on the spinor bundle. Similarly, the Levi-Civita connection of a spin$^r$ manifold together with a connection $\theta$ on the auxiliary bundle defines a connection $\nabla^{\theta}$ on each twisted spin$^r$ spinor bundle. There are obvious analogues of the usual special spinorial field equations (including the classical Killing spinor equation mentioned above) to the spin$^r$ setting.

		\begin{Definition}
			Let $\psi$ be a twisted spin$^r$ spinor on $M$ and $\theta$ a connection on the auxiliary bundle of the spin$^r$ structure.
			\begin{enumerate}\itemsep=0pt
				\item[(1)] $\psi$ is \textit{$\theta$-parallel} if $\nabla^{\theta} \psi = 0$;
				\item[(2)] $\psi$ is \textit{$\theta$-Killing} if for all vector fields $X$ one has that $\nabla_X^{\theta} \psi = \lambda X \cdot \psi$, for some constant~${\lambda \in \mathbb{R}}$;
				\item[(3)] $\psi$ is \textit{$\theta$-generalised Killing} if there exists a symmetric endomorphism field $A\in \End(TM)$ such that, for all vector fields $X$, one has $\nabla_X^{\theta} \psi = A(X) \cdot \psi$.
			\end{enumerate}
		\end{Definition}
		We collect here a number of results that relate the existence of special spin$^r$ spinors to geometric properties of the manifold:
		\begin{Theorem}[\cite{EH16,ES19}] \label{thm:herrera}
			Let $M$ be an $n$-dimensional spin$^r$ manifold, and let $\theta$ be a connection on its auxiliary bundle.
			\begin{enumerate}\itemsep=0pt
				\item[$(1)$] If $M$ carries a $\theta$-parallel spinor $\psi$, then the Ricci tensor decomposes as
				\begin{equation*}
					\Ric = \frac{1}{\abs{\psi}^2} \sum_{k < l} \hat{\Theta}_{kl} \circ \hat{\eta}_{kl}^{\psi} ,
				\end{equation*}
				where $\hat{\Theta}_{kl}$ is the skew-symmetric endomorphism associated to the $2$-form on $TM$ given by~${
				\Theta_{kl} (X,Y) \coloneqq \langle \Omega(X,Y) ( \hat{e}_k ) , \hat{e}_l \rangle } $,
				where $\Omega$ is the curvature $2$-form of the connection $\theta$ on the auxiliary bundle.
				\item[$(2)$] If $\theta$ is flat and $M$ carries a $\theta$-Killing spinor, then $M$ is Einstein.
				\item[$(3)$] If $M$ carries a $\theta$-parallel $m$-twisted pure spinor for some $m \in \mathbb{N}$, $r \geq 3$, $r \neq 4$, $n \neq 8$, $n + 4r - 16 \neq 0$ and $n + 8r - 16 \neq 0$, then $M$ is Einstein.
				\item[$(4)$] If $r=2$ and $M$ carries a $\theta$-parallel pure spinor, then $M$ is K\"{a}hler.
				\item[$(5)$] If $r=3$ and $M$ admits a $\theta$-parallel pure spinor, then $M$ is quaternionic K\"{a}hler.
			\end{enumerate}
		\end{Theorem}
		
		If $M=G/H$ is a Riemannian homogeneous space, invariant connections on homogeneous bundles over $M$ are described by algebraic data \cite{nomizu} (see, e.g., \cite{ANT} for a modern treatment).

		\begin{Proposition}\label{prelims:Nomizu_characterization}
			Let $G/H$ be a homogeneous space, and let $\phi \colon H \to K$ be a Lie group homomorphism. There is a one-to-one correspondence between $G$-invariant connections on \smash{$G\times_{\phi} K$} and linear maps $\Uplambda\colon \mathfrak{g}\to \mathfrak{k}$ satisfying\footnote{Note that $\Ad_H$ in condition (2) refers to the restriction of the adjoint representation of $G$ to $H\subseteq G$, whereas $\Ad_K$ refers to the adjoint representation of $K$.}
			\begin{enumerate}\itemsep=0pt
				\item[$(1)$] $\Uplambda(X) = \phi_* (X)$, $X \in \mathfrak{h}$;
				\item[$(2)$] $\Uplambda \circ \Ad_H(h) = \Ad_K(\phi(h) ) \circ \Uplambda$, $ h \in H$.
			\end{enumerate}
			The map $\Uplambda$ corresponding to a connection is called the \textit{Nomizu map} of said connection.
		\end{Proposition}
		
		For the connections of interest in this article, the Nomizu maps are particularly easy to describe:
		
		\begin{Proposition}[{\cite[Theorem~13.1]{nomizu}}]
			Let $(G/H,g)$ be an $n$-dimensional oriented Riemannian homogeneous space, where the metric $g$ corresponds to an invariant inner product $B$ on a~reductive complement $\mathfrak{m}$ of $\mathfrak{h}$. The Nomizu map $\Uplambda \colon \mathfrak{g} \to \mathfrak{so}(\mathfrak{m})$ of the Levi-Civita connection of~$g$ is given by
			\[
			\Uplambda(X)(Y) = \frac{1}{2} [ X,Y ]_{\mathfrak{m}} + U(X,Y) , \qquad X \in \mathfrak{g} , \quad Y \in \mathfrak{m} ,
			\]
			where $U$ is defined by
			\[
			B ( U(X,Y) , W ) = \frac{1}{2} ( B ( [W,X]_{\mathfrak{m}} , Y ) + B ( X , [W,Y]_{\mathfrak{m}} ) ) .
			\]
		\end{Proposition}

		The following proposition describes how the correspondence in Proposition~\ref{prelims:Nomizu_characterization} works in the particular situation we are interested in.

		\begin{Proposition}\label{prop:lifted_nomizu}
			Let $(G/H,g)$ be an $n$-dimensional Riemannian homogeneous space equipped with a $G$-invariant spin$^r$ structure $P$. Let $\Uplambda \colon \mathfrak{g} \to \mathfrak{so}(n)$ be the Nomizu map of the Levi-Civita connection of $g$, and let $\Uplambda' \colon \mathfrak{g} \to \mathfrak{so}(r)$ be the Nomizu map of an invariant connection $\theta$ on the associated bundle $\hat{P}$. Let \smash{$\widetilde{\Uplambda}$} be the spin lift of $\Uplambda$ to $\mathfrak{spin}(n)$ and let \smash{$\widetilde{\Uplambda'}$} be the spin lift of~$\Uplambda'$ to~$\mathfrak{spin}(r)$. Then, \smash{$\widetilde{\Uplambda} \otimes \bigl(\widetilde{\Uplambda'}\bigr)^{\otimes m}$} is the Nomizu map of the invariant connection $\nabla^{\theta}$ on the $m$-twisted spin$^r$ spinor bundle. Moreover, if $\psi \in (\Sigma_{n,r}^{m})_{\mathrm{inv}}$ and $\hat{X}$ is the fundamental vector field on~$G/H$ defined by $X \in \mathfrak{m}$, then
			\[
			\bigl( \nabla^{\theta}_{\hat{X}} \psi \bigr)_{eH} = \bigl( \widetilde{\Uplambda} \otimes \widetilde{\Uplambda'}^{\otimes m} \bigr) (X)\cdot \psi .
			\]
			In particular, an invariant $m$-twisted spin$^r$ spinor $\psi$ is $\theta$-parallel if, and only if, it satisfies the equation
\smash{$
			\forall X \in \mathfrak{m} \colon \bigl( \widetilde{\Uplambda} \otimes \widetilde{\Uplambda'}^{\otimes m} \bigr) (X)\cdot \psi = 0$}.
		\end{Proposition}
		
		As we shall see later in the setting of spin$^{\C}$ structures, the second condition in Proposition~\ref{prelims:Nomizu_characterization} is quite restrictive. Indeed, the auxiliary bundles of invariant spin$^{\C}$ structures are principal bundles of the abelian group $\SO(2)$. Hence, the second condition becomes $\Uplambda\circ \Ad_H(h) = \Uplambda$ for all $h\in H$. This will force the kernel of $\Uplambda\rvert_{\mathfrak{m}}$ to be quite large in most of our cases. The following is a useful criterion.

		\begin{Lemma}\label{prelims:Nomizu_criterion}
			Let $G/H$ be a homogeneous space with a reductive decomposition $\mathfrak{g} = \mathfrak{h} \oplus \mathfrak{m}$, and let $\phi \colon H \to K$ be a Lie group homomorphism. Let $H_0\subseteq H$ be the kernel of $\Ad_K \circ \phi\colon H \to \GL(\mathfrak{k})$. If $X \in \text{span}_{\R}[ \mathfrak{h}_0 , \mathfrak{m} ]$, then $\Uplambda(X)=0$ for the Nomizu map $\Uplambda\colon\mathfrak{g}\to\mathfrak{k}$ associated to any invariant connection on $G\times_{\phi} K$.
		\end{Lemma}
		\begin{proof}
			By linearity of $\Uplambda$, it suffices to consider $X=[v,Y]$ for some $v \in \mathfrak{h}_0$ and $Y \in \mathfrak{m}$. Let~${\gamma\colon \R \to H_0}$ be a curve with $\gamma(0)=e_{G}$ and $\gamma'(0) = v$. By Proposition~\ref{prelims:Nomizu_characterization}, the Nomizu map~$\Uplambda$ of any invariant connection on $G\times_{\phi} K$ satisfies $\Uplambda\circ \Ad_H(\gamma(t)) = \Uplambda$, and hence
			\[0= \frac{{\rm d}}{{\rm d}t}\big\rvert_{t=0} \Uplambda(Y) = \frac{\rm d}{{\rm d}t}\big\rvert_{t=0} \Uplambda(\Ad_H(\gamma(t)) Y) = \Uplambda([v,Y]) = \Uplambda(X).\tag*{\qed}
 \] \renewcommand{\qed}{}
		\end{proof}

Finally, we examine the differential equations satisfied by invariant spin$^r$ spinors on symmetric spaces. The following proposition is analogous to the familiar fact in the spin setting that invariant spinors on symmetric spaces are $\nabla^g$-parallel, since the Levi-Civita and the Ambrose--Singer connections coincide.

		\begin{Proposition}\label{symmetricspace_differential_equation}
			Let $(M=G/H,g)$ be a Riemannian symmetric space admitting a $G$-invariant spin$^r$ structure $P$. Let $\mathfrak{g} = \mathfrak{h} \oplus \mathfrak{m}$ be a reductive decomposition such that ${[\mathfrak{m},\mathfrak{m}] \subseteq \mathfrak{h}}$. Let~$E$ be the natural vector bundle associated to the auxiliary bundle $\hat{P}$, and let $\nabla^E$ be the unique $G$-invariant connection on $E$ whose Nomizu map vanishes identically on $\mathfrak{m}$. Let $\nabla:= \nabla^g \otimes \bigl(\nabla^E\bigr)^{\otimes m}$ be the corresponding twisted connection on $\Sigma_{n,r}^m M$.
			
			If $\psi \in \Sigma_{n,r}^m$ is a $G$-invariant spin$^r$ spinor, then $\nabla \psi =0$.
		\end{Proposition}
		\begin{proof}
			With respect to the reductive decomposition $\mathfrak{g} = \mathfrak{h} \oplus \mathfrak{m}$, the Nomizu map associated to the Levi-Civita connection vanishes identically on the reductive complement $\mathfrak{m}$, i.e., $\Uplambda^g\rvert_{\mathfrak{m}} \equiv 0$. The Nomizu map of $\nabla$ then vanishes identically on $\mathfrak{m}$, and the result follows.
		\end{proof}
		
		This result will be useful for several of the cases in our classification, where the limited number of low-dimensional representations of the isotropy groups will force the auxiliary bundles to be isomorphic to familiar tensor (sub)bundles.

		\section{Projective spaces}\label{section:projective}
		
		Oni\v{s}\v{c}ik \cite[p.\ 163]{oni} classified the compact, simple and simply connected Lie groups which act transitively on the projective spaces $\mathbb{CP}^n$, $\mathbb{HP}^n$ and $\mathbb{OP}^2$ -- see also \cite[p.\ 356]{ziller}. We exhibit them in Table~\ref{table:groups}.
		
		\begin{table}[h!]
			\centering
			\begin{tabular}{ c c c }
				\toprule
				Space & $G$ & $H$ \\
				\toprule
				$\mathbb{CP}^{n}$ & $\SU(n+1)$ & $\Special ( \U(1) \times \U(n) )$ \\
				$\mathbb{CP}^{2n+1}$ & $\Sp(n+1)$ & $\U(1) \times \Sp(n)$ \\
				\midrule
				$\mathbb{HP}^{n}$ & $\Sp(n+1)$ & $\Sp(1) \times \Sp(n)$ \\
				\midrule
				$\mathbb{OP}^{2}$ & $F_4$ & $\Spin(9)$ \\
				\bottomrule
			\end{tabular}
			\caption{Compact, simple and simply connected Lie groups $G$ acting transitively with isotropy $H$ on projective spaces -- see, e.g., \cite[p.~356]{ziller}. }
			\label{table:groups}
		\end{table}
		
		\subsection{Hermitian complex projective space}\label{section:her_cpn}
		
		In this section, we consider the complex projective space realised as the quotient
		\[
		\mathbb{CP}^n \cong \faktor{\SU(n+1)}{\Special ( \U(1) \times \U(n) )} ,
		\]
		where
		\begin{equation*}
			\Special (\U(1) \times \U(n) ) = \left\{
			\begin{pmatrix}
				z
				& \rvline & 0 \\
				\hline
				0 & \rvline &
				B
			\end{pmatrix}
			\in \Mat_{n+1} ( \mathbb{C} ) \mid z \in \U(1) , \, B \in \U(n), \, z \det(B) = 1
			\right\} .
		\end{equation*}
		In order to study $\SU(n+1)$-invariant spin$^r$ structures and spinors on this space, we need to establish some notation and properties of the Lie algebras involved. Let us denote by $\mathfrak{h}$ the Lie algebra of $H \coloneqq \Special (\U(1) \times \U(n) )$, and consider the copy of $\SU(n)$ included in $\SU(n+1)$ as the lower right-hand $n \times n$ block. Letting $\mathfrak{h}' \coloneqq \mathfrak{su}(n) \subseteq \mathfrak{su}(n+1)$, we have the decomposition
		\begin{equation*}
			\mathfrak{h} = \mathbb{R}\xi \oplus \mathfrak{h}' \qquad \text{(as Lie algebras), }
		\end{equation*}
		where \smash{$\xi \coloneqq {\rm i} \bigl( -n F_{1,1}^{(n+1)} + \sum_{l=2}^{n+1}F_{l,l}^{(n+1)}\bigr)$} and
		\begin{align*}
			&\mathfrak{h}' = \mathfrak{su}(n) = \vecspan_{\mathbb{R}} \bigl\{ {\rm i} F_{p,q}^{(n+1)}, E_{p,q}^{(n+1)}, {\rm i}\bigl( F_{r,r}^{(n+1)} - F_{r+1,r+1}^{(n+1)} \bigr)\bigr\}_{\begin{subarray}{l} 2 \leq p < q \leq n+1 \\ r = 2 , \dots, n \end{subarray}} .
		\end{align*}
		The isotropy subalgebra $\mathfrak{h} \subseteq \mathfrak{su}(n+1)$ has a reductive complement
		\begin{equation*}
			\mathfrak{m} \coloneqq ( \mathfrak{h})^{\perp_{B_0}} = \vecspan_{\mathbb{R}} \bigl\{ {\rm i} F_{1,p}^{(n+1)}, E_{1,p}^{(n+1)}\bigr\}_{p = 2, \dots, n+1} ,
		\end{equation*}
		and the adjoint representation of $\mathfrak{h}$ on $\mathfrak{m}$ is irreducible. Hence, by Theorem~\ref{thm:inv_metrics}, the $\SU(n+1)$-invariant metrics on $\mathbb{CP}^n$ correspond to the inner products on $\mathfrak{m}$ in the one-parameter family~${
			g_a \coloneqq a B_0 \restr{\mathfrak{m} \times \mathfrak{m}}}$, $ a >0 $,
		and a $g_a$-orthonormal basis of $\mathfrak{m}$ is given by
		\begin{equation*}
			\left\{ e_{2p-1} \coloneqq \frac{{\rm i}}{\sqrt{2a}} F_{1,p+1}^{(n+1)} , e_{2p} \coloneqq \frac{1}{\sqrt{2a}} E_{1,p+1}^{(n+1)} \right\}_{p=1, \dots, n} .
		\end{equation*}
		We take the orientation defined by the ordering $(e_1,e_2,\dots,e_{2n-1},e_{2n})$.

		\subsubsection[Invariant spin\^{}r structures]{Invariant spin$^{\boldsymbol{r}}$ structures}
		
		We are now ready to determine the $\SU(n+1)$-invariant spin type of $\mathbb{CP}^{n}$. By \cite[p.~327]{HS90}, it is clear that $\mathbb{CP}^{n}$ admits an $\SU(n+1)$-invariant spin structure if, and only if, $n$ is odd. Moreover, one has the following.

		\begin{Theorem}\label{thm:her_cpn_str}
			The $\SU(n+1)$-invariant spin$^{\mathbb{C}}$ structures on $\mathbb{CP}^n$ are given by
			\begin{equation*}
				\SU(n+1) \times_{\phi_s} \Spin^{\mathbb{C}}(2n) , \qquad s \in \mathbb{Z} \colon\ n \not\equiv s \text{ \rm mod } 2 ,
			\end{equation*}
			where $\phi_s$ is the unique lift of $\sigma \times \varphi_s$ to $\Spin^{\mathbb{C}}(2n)$, $\sigma \colon H \to \SO(2n)$ is the isotropy representation and $\varphi_s \colon H \to \SO(2) \cong \U(1)$ is given by
\[
\begin{pmatrix}
				z
				& \rvline & 0 \\
				\hline
				0 & \rvline &
				B
			\end{pmatrix} \mapsto \det(B)^s.
\]
 In particular, the $\SU(n+1)$-invariant spin type of $\mathbb{CP}^{n}$ is
			\begin{equation*}
				\Sigma(\mathbb{CP}^n,\SU(n+1)) =
				\begin{dcases}
					1, & n \text{ odd, }\\
					2, & n \text{ even. }
				\end{dcases}
			\end{equation*}
		\end{Theorem}
		\begin{proof}
			Note that $H \cong \U(n)$, and that every Lie group homomorphism $\U(n) \to \U(1)$ is of the form $B \mapsto \det(B)^s$, for some $s \in \mathbb{Z}$. The loop $\alpha(t) = \diag \bigl( {\rm e}^{- 2 \pi {\rm i} t} , 1 , \dots , 1 , {\rm e}^{2 \pi {\rm i} t} \bigr)$ generates~${\pi_1 ( H ) \cong \Z}$, and
			\[
			( \sigma \times \varphi_s )_{\sharp} (\alpha) = ( n-1 , s ) \in \pi_1 ( \SO(2n) ) \times \pi_1 ( \SO(2) ) .
			\]
			This can be seen as follows: the image of $\alpha (t)$ under the isotropy representation $\sigma$ is easily seen to be
			\begin{equation*}
				\sigma ( \alpha (t) ) = \diag \bigl( {\rm e}^{2 \pi {\rm i} t} , \dots , {\rm e}^{2 \pi {\rm i} t} , {\rm e}^{4 \pi {\rm i} t} \bigr) \in \U(n) \subseteq \SO(2n) ,
			\end{equation*}
			where ${\rm e}^{2 \pi {\rm i} t}$ appears $n-1$ times. This can be seen using the realisation of $\sigma$ as the action of $H$ on $\mathfrak{m}$ by matrix conjugation.
			
			Hence, by Proposition~\ref{prop:fg}, $\sigma \times \varphi_s \colon H \to \SO(2n) \times \U(1)$ lifts to $\Spin^{\mathbb{C}}(2n)$ if, and only if, $n \not\equiv s$ mod $2$. Finally, as $\U(1)$ is abelian, the representations $\varphi_s$ are pairwise non-equivalent. The result now follows from Theorem~\ref{cor:inv_gen_spin}.
		\end{proof}
		
		\subsubsection[Invariant spin\^{}r spinors]{Invariant spin$^{\boldsymbol{r}}$ spinors}\label{section:her_cpn:spinors}
		
		The classical spin case $r=1$ does not yield any non-trivial invariant spinors, as we show in the following theorem.
		\begin{Theorem} \label{thm:her_classical_spinors}
			For $n$ odd, there are no non-trivial $\SU(n+1)$-invariant spinors on $\mathbb{CP}^{n}$.
		\end{Theorem}
		\begin{proof}
			We need the explicit expression of the action of $\xi \in \mathfrak{h}$ on $\mathfrak{m}$. Letting $\sigma \colon H \to \SO(2n)$ be the isotropy representation and $\widetilde{\sigma}$ its lift to $\Spin(2n)$, and, e.g., using the commutation relations in \cite[p.\ 9]{AHL}, one can readily see that, for each $p = 1, \dots, n$,
			\begin{equation*}
				\ad(\xi)\restr{\mathfrak{m}} (e_{2p}) = [\xi,e_{2p} ]_{\mathfrak{m}} = (n+1) e_{2p-1} , \qquad \ad(\xi)\restr{\mathfrak{m}} (e_{2p-1}) = [\xi,e_{2p-1} ]_{\mathfrak{m}} = - (n+1) e_{2p} .
			\end{equation*}
			Hence,
			\begin{equation*}
				\sigma_* (\xi) = \ad(\xi)\restr{\mathfrak{m}} = (n+1) \sum_{p=1}^{n} e_{2p}\wedge e_{2p-1} \in \mathfrak{so}(2n) ,
			\end{equation*}
			and the spin lift is given by
			\begin{equation*}
				\widetilde{\sigma}_*(\xi) = \frac{n+1}{2} \sum_{p=1}^{n} e_{2p} \cdot e_{2p-1} \in \mathfrak{spin}(2n) \subseteq \mathbb{C}\mathrm{l}(2n) .
			\end{equation*}
			A direct computation using \eqref{eq:def_action} shows that, for each $1 \leq k \leq n$ and $1 \leq j_1 < \dots < j_k \leq n$,
			\[%\label{eq:su_computation}
				\widetilde{\sigma}_*(\xi) \cdot ( y_{j_1} \wedge \cdots \wedge y_{j_k} ) = \frac{i(n+1)}{2}(2k-n) y_{j_1} \wedge \cdots \wedge y_{j_k} .
			\]
			From this we observe that, if $n$ is odd, there are no non-trivial invariant spinors.
		\end{proof}
		
		The fact that no non-trivial invariant spinors exist motivates the investigation of spin$^{\C}$ spinors.
		\begin{Theorem}\label{thm:inv_gen_cpn}
			For $n,s \in \mathbb{N}$ with $n \not\equiv s$ {\rm mod} $2$, the space of $\SU(n+1)$-invariant $1$-twisted spin$^{\mathbb{C}}$ spinors on $\mathbb{CP}^n$ associated to the spin$^{\mathbb{C}}$ structure $\SU(n+1) \times_{\phi_s} \Spin^{\mathbb{C}}(2n)$ is given by
			\begin{equation*}
				\bigl(\Sigma_{2n,2}^{1}\bigr)_{\mathrm{inv}} =
				\begin{dcases}
					\vecspan_{\mathbb{C}} \{ 1 \otimes \hat{1} , (y_1 \wedge \cdots \wedge y_n) \otimes \hat{y}_1 \}, & s = n+1 , \\
					\vecspan_{\mathbb{C}} \{ (y_1 \wedge \cdots \wedge y_n) \otimes \hat{1} , 1 \otimes \hat{y}_1 \}, & s = -(n+1) , \\
					0 ,& \text{otherwise} .
				\end{dcases}
			\end{equation*}
			In particular, the $\SU(n+1)$-invariant spinor type of $\mathbb{CP}^n$ is
$
			\sigma (\mathbb{CP}^n, \SU(n+1)) = 2 $.
		\end{Theorem}
		
		\begin{proof}
			Recall that $\mathfrak{h} = \mathbb{R}\xi \oplus \mathfrak{h}'$ as Lie algebras, and note that, for $\psi \in \Sigma_{2n,2}^{1}$,
			\[
			( \forall X \in \mathfrak{h}' \colon ( \phi_s )_* (X) \cdot \psi = 0 ) \iff \psi \in \vecspan_{\mathbb{C}} \{1,y_1 \wedge \cdots \wedge y_n \} \otimes \Sigma_2 ,
			\]
			by \cite[Theorem~3.7]{AHL} and the definition of $\varphi_s$. Moreover,
			\[
			(\phi_s)_* (\xi) = \Bigg( \frac{n+1}{2} \sum_{p=1}^{n} e_{2p} \cdot e_{2p-1} , \frac{sn}{2} \hat{e}_1 \cdot \hat{e}_2 \Bigg) \in \mathfrak{spin}(2n) \oplus \mathfrak{spin}(2) \cong \mathfrak{spin}^{\C}(2n) .
			\]
			Finally, an easy calculation shows that, for $0 \leq k \leq n$ and $1 \leq j_1 < \dots < j_k \leq n$,
			\begin{align*}
				&(\phi_s)_* (\xi) \cdot \bigl( ( y_{j_1} \wedge \cdots \wedge y_{j_k} ) \otimes \hat{1} \bigr) = \frac{{\rm i}}{2} ( (n+1)(2k-n) + sn ) \bigl( ( y_{j_1} \wedge \cdots \wedge y_{j_k} ) \otimes \hat{1} \bigr) , \\ &(\phi_s)_* (\xi) \cdot ( ( y_{j_1} \wedge \cdots \wedge y_{j_k} ) \otimes \hat{y}_1 ) = \frac{{\rm i}}{2} ( (n+1)(2k-n) - sn ) ( ( y_{j_1} \wedge \cdots \wedge y_{j_k} ) \otimes \hat{y}_1 ) .
			\end{align*}
			From this it is straightforward to conclude the result.
		\end{proof}

		\subsubsection[Special spin\^{}r spinors]{Special spin$^{\boldsymbol{r}}$ spinors}
		
		The aim now is to show that the $\SU(n+1)$-invariant spin$^{\mathbb{C}}$ spinors on $\mathbb{CP}^n$ found in Theorem~\ref{thm:inv_gen_cpn} are pure and parallel with respect to a suitable connection on the auxiliary bundle.
		
		\begin{Proposition}\label{lemma:her_cpn_pure_lemma}
			For $s = n+1$ $($resp.\ $s = -(n+1))$, the $\SU(n+1)$-invariant spin$^{\C}$ spinors~${1 \otimes \hat{1}}$ and $(y_1 \wedge \cdots \wedge y_n) \otimes \hat{y}_1$ $\big($resp.\ $(y_1 \wedge \cdots \wedge y_n) \otimes \hat{1}$ and $1 \otimes \hat{y}_1\big)$ on $\mathbb{CP}^{n}$ are pure. Moreover, they are parallel with respect to the invariant connection on the corresponding auxiliary bundle determined by the Nomizu map $\Uplambda \rvert_{\mathfrak{m}} = 0$.
		\end{Proposition}
		\begin{proof}
			We will only show the calculations for the spinor $\psi = (y_1 \wedge \cdots \wedge y_n) \otimes \hat{1}$, as the other three are analogous. Since $r=2<3$, we only need to show that \smash{$\bigl(\hat{\eta}^{\psi}_{12}\bigr)^2 = - \Id$}. Indeed, a~straightforward calculation shows that
			\begin{gather*}
				\eta^{\psi}_{12} (e_{2p}, e_{2q}) = \Re \langle e_{2p} \cdot e_{2q} \cdot \hat{e}_1 \cdot \hat{e}_2 \cdot \psi , \psi \rangle + \delta_{p,q} \Re \langle \hat{e}_1 \cdot \hat{e}_2 \cdot \psi , \psi \rangle \\
			\phantom{\eta^{\psi}_{12} (e_{2p}, e_{2q}) }{}	 = \Re \langle {\rm i} e_{2p} \cdot e_{2q} \cdot \psi , \psi \rangle + \delta_{p,q}\Re \langle -{\rm i} \psi , \psi \rangle = 0, \\
				\eta^{\psi}_{12} (e_{2p-1}, e_{2q-1}) = 0, \\
				\eta^{\psi}_{12} (e_{2p}, e_{2q-1}) = \Re \langle e_{2p} \cdot e_{2q-1} \cdot \hat{e}_1 \cdot \hat{e}_2 \cdot \psi , \psi \rangle = \Re \langle {\rm i} e_{2p} \cdot e_{2q-1} \cdot \psi , \psi \rangle = - \delta_{p,q}.
			\end{gather*}
			Hence, $\hat{\eta}^{\psi}_{12} (e_{2p}) = - e_{2p-1}$ and $\hat{\eta}^{\psi}_{12} (e_{2p-1}) = e_{2p}$.
			
			The last assertion of the proposition follows by noting that, as $\C\mathbb{P}^n \cong \SU(n+1)/\Special(\U(1)\times \U(n))$ is a symmetric space, the Levi-Civita connection coincides with the Ambrose--Singer connection, whose Nomizu map satisfies $\Uplambda^g\rvert_{\mathfrak{m}} \equiv 0$.
		\end{proof}
		
		In light of the Ricci decomposition in \cite[Theorem~3.1]{EH16}, the existence of parallel pure spin$^{\C}$ spinors encodes a very well-known fact -- see, e.g., \cite{ziller}.
		
		\begin{Theorem}\label{cor:cpn_conclusion}
			The $\SU(n+1)$-invariant metrics $g_a$ on $\mathbb{CP}^n$ are K\"{a}hler--Einstein.
		\end{Theorem}
		\begin{proof}
			Take $s=-(n+1)$. Consider the spin$^{\C}$ structure defined by $\varphi_s$, and endow its auxiliary bundle with the connection described in Proposition~\ref{lemma:her_cpn_pure_lemma}. We have seen that this spin$^{\C}$ structure carries a non-zero parallel pure spin$^{\C}$ spinor $\psi = (y_1 \wedge \cdots \wedge y_n) \otimes \hat{1} $. This implies that the metric is K\"{a}hler \cite[Corollary~4.10]{ES19} with respect to the invariant complex structure defined by~$\hat{\eta}^{\psi}_{12}$. Now, by Theorem~\ref{thm:herrera}\,(1), the Ricci tensor decomposes as
\[
			\Ric = \frac{1}{\abs{\psi}^2} \hat{\Theta}_{1 2} \circ \hat{\eta}_{1 2}^{\psi} ,
\]
			where $\hat{\Theta}_{1 2}$ is the endomorphism associated to the $2$-form on $\mathfrak{m}$
			$
			\Theta_{1 2} (X,Y) \coloneqq \langle \Omega(X,Y) ( \hat{e}_1 ) , \hat{e}_2 \rangle $, $ X,Y \in \mathfrak{m}$,
			where~$\Omega$ is the curvature $2$-form of the connection on the auxiliary bundle. Recall \cite{ANT} that, if $\Uplambda$ is the Nomizu map of the connection on the auxiliary bundle, then
			\[
			\forall X,Y \in \mathfrak{m} \colon\ \Omega(X,Y) = [ \Uplambda(X), \Uplambda(Y) ]_{\mathfrak{so}(2)} - \Uplambda( [X,Y] ) = - \Uplambda( [X,Y] ) .
			\]
			It is now easy to see that, for all $1 \leq p, q \leq n$,
			\[
			\Omega( e_{2p-1}, e_{2q} ) = \delta_{p,q} \frac{s}{a} \hat{e}_1 \wedge \hat{e}_2 , \qquad \Omega( e_{2p-1}, e_{2q-1} ) = \Omega( e_{2p}, e_{2q} ) = 0 .
			\]
			Hence,
			\[
			\hat{\Theta}_{1 2} = \frac{s}{a} \sum_{p=1}^n e_{2p-1} \wedge e_{2p} ,
			\]
			and finally, using the expression of $\hat{\eta}^{\psi}_{12}$ obtained in the proof of Proposition~\ref{lemma:her_cpn_pure_lemma}, we obtain
			\begin{equation}\label{eq:ricci_cpn}
				\Ric = \frac{1}{\abs{\psi}^2} \hat{\Theta}_{1 2} \circ \hat{\eta}_{1 2}^{\psi} = \frac{n+1}{a} \Id .
			\end{equation}
This proved the theorem.
		\end{proof}

		\begin{Remark}
			Recall that the Fubini--Study metric $g_{\rm FS}$ on $\mathbb{CP}^n$ is $\SU(n+1)$-invariant, and that its Ricci constant is $2(n+1)$. From equation \eqref{eq:ricci_cpn}, we can deduce that $g_{\rm FS}=g_{1/2}$.
		\end{Remark}
		
		\subsection{Symplectic complex projective space}
		
		Consider, for $n \geq 1$, the homogeneous realisation of odd-dimensional complex projective space
		\[
		\mathbb{CP}^{2n+1} \cong \faktor{\Sp(n+1)}{ \U(1) \times \Sp(n)} ,
		\]
where $H \coloneqq \U(1) \times \Sp(n)$ is realised as a subgroup of $\Sp(n+1)$ by the upper left-hand $1 \times 1$ block for $\U(1)$ and the lower right-hand $n \times n$ block for $\Sp(n)$. Denote by $\mathfrak{h}$ the Lie algebra of~$H$ and $\mathfrak{h}' \coloneqq \mathfrak{sp}(n) \subseteq \mathfrak{sp}(n+1)$. Then,
	$
			\mathfrak{h} = \mathbb{R}\xi_1 \oplus \mathfrak{h}' $ (as Lie algebras),
		where \smash{$\xi_1 \coloneqq {\rm i} F_{1,1}^{(n+1)}$} and
			\begin{align*}
\mathfrak{h}' &= \mathfrak{sp}(n) \\
&= \vecspan_{\mathbb{R}} \bigl\{ {\rm i} F_{p,p}^{(n+1)}, {\rm j} F_{p,p}^{(n+1)}, {\rm k} F_{p,p}^{(n+1)}, {\rm i} F_{r,s}^{(n+1)}, {\rm j} F_{r,s}^{(n+1)}, {\rm k} F_{r,s}^{(n+1)}, E_{r,s}^{(n+1)} \bigr\}_{\begin{subarray}{l} 2 \leq r < s \leq n+1 \\ p = 2 , \dots, n+1 \end{subarray}} .
\end{align*}
		The Lie subalgebra $\mathfrak{h} \subseteq \mathfrak{sp}(n+1)$ has a reductive complement $\mathfrak{m} \coloneqq ( \mathfrak{h})^{\perp_{B_0}} = \mathcal{V} \oplus \mathcal{H}$, where
		\begin{align*}
			&\mathcal{V} \coloneqq \vecspan_{\mathbb{R}} \bigl\{ \xi_2 \coloneqq -{\rm k} F_{1,1}^{(n+1)}, \xi_3 \coloneqq {\rm j} F_{1,1}^{(n+1)} \bigr\} , \\
			&\mathcal{H} \coloneqq \vecspan_{\mathbb{R}} \bigl\{e_{4p} \coloneqq {\rm j} F_{1,p+1}^{(n+1)}, e_{4p+1} \coloneqq {\rm k} F_{1,p+1}^{(n+1)}, e_{4p+2} \coloneqq {\rm i} F_{1,p+1}^{(n+1)}, e_{4p+3} \coloneqq E_{1,p+1}^{(n+1)}\bigr\}_{p = 1, \dots, n} ,
		\end{align*}
		and this is the decomposition of $\mathfrak{m}$ into irreducible subrepresentations\footnote{Note that this decomposition into $\mathcal{V}$ and $\mathcal{H}$ corresponds to the vertical and horizontal distributions of the generalised Hopf fibration $S^2 \hookrightarrow \mathbb{CP}^{2n+1} \to \mathbb{HP}^n$.} of the adjoint representation of $\mathfrak{h}$. We have, therefore, by Theorem~\ref{thm:inv_metrics}, a two-parameter family of invariant metrics
		\begin{equation*}
			g_{a,t} \coloneqq a B_0\restr{\mathcal{H} \times \mathcal{H}} + 2at B_0\restr{\mathcal{V} \times \mathcal{V}} , \qquad a,t>0 ,
		\end{equation*}
		and a $g_{a,t}$-orthonormal basis of $\mathfrak{m}$ is given by
		\[
		\left\{ \xi_2^{a,t} \coloneqq \frac{1}{\sqrt{2ta}} \xi_2 , \xi_3^{a,t} \coloneqq \frac{1}{\sqrt{2ta}} \xi_3 , e_{4p + \varepsilon}^{a,t} \coloneqq \frac{1}{\sqrt{2a}} e_{4p + \varepsilon} \right\}_{\begin{subarray}{l} \varepsilon = 0,1,2,3 \\ p = 1 , \dots, n \end{subarray}} .
		\]
		We take the orientation defined by the ordering \smash{$\bigl(\xi^{a,t}_2, \xi^{a,t}_3, e^{a,t}_4, \dots, e^{a,t}_{4n+3}\bigr)$}.

		\subsubsection[Invariant spin\^{}r structures]{Invariant spin$^{\boldsymbol{r}}$ structures}\label{inv_spinr_str_SpnU1}
		
		We begin by determining the $\Sp(n+1)$-invariant spin type of $\mathbb{CP}^{2n+1}$. By \cite[p.\ 327]{HS90}, it is clear that $\mathbb{CP}^{2n+1}$ admits a unique spin structure, and this structure is $\Sp(n+1)$-invariant. Using the algebraic characterisation in Theorem~\ref{cor:inv_gen_spin}, we can explicitly obtain all $\Sp(n+1)$-invariant spin$^{\C}$ structures on $\mathbb{CP}^{2n+1}$:
		
		\begin{Theorem}\label{thm_inv_gen_spin_structures_SpnU1}
			The $\Sp(n+1)$-invariant spin$^{\mathbb{C}}$ structures on $\mathbb{CP}^{2n+1}$ are given by
			\begin{equation*}
				\Sp(n+1) \times_{\phi_s} \Spin^{\mathbb{C}}(4n+2) , \qquad s \in 2\Z ,
			\end{equation*}
			where $\phi_s$ is the unique lift of $\sigma \times \varphi_s$ to $\Spin^{\mathbb{C}}(4n+2)$, $\sigma \colon H \to \SO(4n+2)$ is the isotropy representation and $\varphi_s \colon H \to \SO(2) \cong \U(1)$ is defined by $(z,A) \mapsto z^s$.
		\end{Theorem}
		
		\begin{proof}
			This follows from Theorem~\ref{cor:inv_gen_spin}, together with the fact that $\Sp(n)$ is simple and that all Lie group homomorphisms $\U(1) \to \U(1)$ are of the form $z \mapsto z^s$, for some $s \in \mathbb{Z}$. Using Proposition~\ref{prop:fg} as in the proof of Theorem~\ref{thm:her_cpn_str}, one sees that $\sigma \times \varphi_s$ lifts to $\Spin^{\mathbb{C}}(4n+2)$ if, and only if, $s$ is even. As $\U(1)$ is abelian, the representations $\varphi_s$ are pairwise non-equivalent. Hence, these spin$^{\mathbb{C}}$ structures are pairwise non-$\Sp(n+1)$-equivariantly equivalent.
		\end{proof}
		
		\subsubsection[Invariant spin\^{}r spinors]{Invariant spin$^{\boldsymbol{r}}$ spinors}
		
		First, we classify the $\Sp(n+1)$-invariant spinors for the unique spin structure of $\mathbb{CP}^{2n+1}$.
		\begin{Theorem}\label{thm:symp_cpn_spinors}
			The space $\Sigma_{\mathrm{inv}}$ of $\Sp(n+1)$-invariant spinors on $\mathbb{CP}^{2n+1}$ is given by
			\begin{equation*}
				\Sigma_{\text{inv}} =
				\begin{dcases}
					\vecspan_{\mathbb{C}} \bigl\{\psi_{+} \coloneqq \omega^{(n+1)/2} , \psi_{-} \coloneqq y_1 \wedge \omega^{(n-1)/2} \bigr\}, & n \text{ odd, }\\
					0 ,& n \text{ even, }
				\end{dcases}
			\end{equation*}
			where $\omega \coloneqq \sum_{p=1}^{n} y_{2p} \wedge y_{2p+1}$.
		\end{Theorem}
		\begin{proof}
			By \cite[Theorem~4.11]{AHL}, the space of invariant spinors is quite restricted:
			\[
			\Sigma_{\mathrm{inv}} \subseteq \vecspan_{\mathbb{C}}\bigl\{ \omega^k , y_1 \wedge \omega^k \bigr\}_{k=0, \dots , n } .
			\]
			We only need to determine which of these are annihilated by $\xi_1$. A computation analogous to the one in the proof of Theorem~\ref{thm:her_classical_spinors} shows that, if $\widetilde{\sigma}$ is the lift to $\Spin(4n+2)$ of the isotropy representation $\sigma \colon H \to \SO(4n+2)$,
			\[ \widetilde{\sigma}_*(\xi_1) = \xi^{a,t}_2 \cdot \xi^{a,t}_3 + \frac{1}{2} \sum_{p=1}^{n} \bigl( e^{a,t}_{4p} \cdot e^{a,t}_{4p+1} + e^{a,t}_{4p+2} \cdot e^{a,t}_{4p+3} \bigr) . \]
			In particular,
			\[ \widetilde{\sigma}_*(\xi_1) \cdot \omega^k = {\rm i} (n+1-2k) \omega^k , \qquad \widetilde{\sigma}_* (\xi_1) \cdot \bigl( y_1 \wedge \omega^k \bigr) = {\rm i} (n-1-2k)\bigl( y_1 \wedge \omega^k \bigr) , \]
			and the result follows.
		\end{proof}

		We now turn to the study of spin$^{\mathbb{C}}$ spinors. Using the same argument as in the Hermitian case (see Theorem~\ref{thm:inv_gen_cpn}), one obtains.

		\begin{Theorem} \label{thm:inv_gen_cpn2}
			For $n \in \N$ and $s=2s' \in 2\Z$, the space \smash{$\bigl(\Sigma_{4n+2,2}^{1}\bigr)_{\mathrm{inv}}$} of $\Sp(n+1)$-invariant $1$-twisted spin$^{\mathbb{C}}$ spinors on $\mathbb{CP}^{2n+1}$ associated to the spin$^{\mathbb{C}}$ structure $\Sp(n+1) \times_{\phi_s} \Spin^{\mathbb{C}}(4n+2)$ is given by
			\begin{equation*}
				\bigl(\Sigma_{4n+2,2}^{1}\bigr)_{\mathrm{inv}} =
				\begin{dcases}
					\vecspan_{\mathbb{C}} \bigl\{ \omega^{(n+1+s')/2} , y_1 \wedge \omega^{(n-1+s')/2} \bigr\} \otimes \hat{1} \oplus \\
					\qquad \oplus \vecspan_{\mathbb{C}}\bigl\{ \omega^{(n+1-s')/2} , y_1 \wedge \omega^{(n-1-s')/2} \bigr\} \otimes \hat{y_1}, & n \not\equiv s' \mod 2 , \\
					0, & \text{otherwise, }
				\end{dcases}
			\end{equation*}
			where negative powers of $\omega$ are defined to be zero. In particular, the $\Sp(n+1)$-invariant spinor type of $\mathbb{CP}^{2n+1}$ satisfies
			\begin{equation*}
				\sigma\bigl( \mathbb{CP}^{2n+1} , \Sp(n+1) \bigr) =
				\begin{dcases}
					1 , & n\  \text{odd}, \\
					2 ,& n\  \text{even}.
				\end{dcases}
			\end{equation*}
		\end{Theorem}
		\begin{Remark}
			The spin$^{\C}$ structure corresponding to $s=0$ is the one induced by the usual spin structure. Indeed, taking $s=0$ in Theorem~\ref{thm:inv_gen_cpn2}, one recovers the spinors in Theorem~\ref{thm:symp_cpn_spinors} tensored with $\Sigma_2$.
		\end{Remark}
		
		\subsubsection[Special spin\^{}r spinors]{Special spin$^{\boldsymbol{r}}$ spinors}
		
		In order to differentiate these spinors, one can see, using the formulas for the Nomizu map from~\cite[p.\ 141]{BFGK}, that the spin lift $\widetilde{\Uplambda}^{a,t}$ of the Nomizu map of the Levi-Civita connection of $g_{a,t}$ is given by
		\begin{gather}
				\widetilde{\Uplambda}^{a,t} \bigl( \xi_2^{a,t} \bigr) = \frac{1-t}{2\sqrt{2at}} \sum_{p=1}^n \bigl( e_{4p}^{a,t} \cdot e_{4p+2}^{a,t} - e_{4p+1}^{a,t} \cdot e_{4p+3}^{a,t} \bigr) ,\nonumber \\
				\widetilde{\Uplambda}^{a,t} \bigl( \xi_3^{a,t} \bigr) = \frac{1-t}{2\sqrt{2at}} \sum_{p=1}^n \bigl( e_{4p}^{a,t} \cdot e_{4p+3}^{a,t} + e_{4p+1}^{a,t} \cdot e_{4p+2}^{a,t} \bigr) ,\nonumber \\
				\widetilde{\Uplambda}^{a,t} \bigl( e_{4p}^{a,t} \bigr) = \frac{1}{2}\sqrt{\frac{t}{2a}} \bigl( - \xi_2^{a,t} \cdot e_{4p+2}^{a,t} - \xi_3^{a,t} \cdot e_{4p+3}^{a,t} \bigr) , \nonumber\\
				\widetilde{\Uplambda}^{a,t} \bigl( e_{4p+1}^{a,t} \bigr) = \frac{1}{2}\sqrt{\frac{t}{2a}} \bigl( \xi_2^{a,t} \cdot e_{4p+3}^{a,t} - \xi_3^{a,t} \cdot e_{4p+2}^{a,t} \bigr) ,\nonumber\\
				\widetilde{\Uplambda}^{a,t} \bigl( e_{4p+2}^{a,t} \bigr) = \frac{1}{2}\sqrt{\frac{t}{2a}} \bigl( \xi_2^{a,t} \cdot e_{4p}^{a,t} + \xi_3^{a,t} \cdot e_{4p+1}^{a,t} \bigr) , \nonumber\\
			\widetilde{\Uplambda}^{a,t} \bigl( e_{4p+3}^{a,t} \bigr) = \frac{1}{2}\sqrt{\frac{t}{2a}} \bigl( - \xi_2^{a,t} \cdot e_{4p+1}^{a,t} + \xi_3^{a,t} \cdot e_{4p}^{a,t} \bigr) .\label{eq:symp_nomizu}
		\end{gather}
		
		Baum et al.\ proved in \cite[p.\ 146]{BFGK} that $\mathbb{CP}^3$ admits non-trivial $\Sp(2)$-invariant generalised Killing spinors, given by $\psi_{+} \pm {\rm i} \psi_{-}$. For the Fubini--Study metric, the two eigenvalues of these generalised Killing spinors coincide, yielding real Killing spinors which are related to the nearly K\"{a}hler geometry of $\mathbb{CP}^3$. We now show that this does not occur in higher dimensions.
		\begin{Theorem}
			The spaces $\mathbb{CP}^{2n+1}$ admit non-trivial $\Sp(n+1)$-invariant generalised Killing spinors if, and only if, $n=1$.
		\end{Theorem}
		\begin{proof}
			First, we recall from Theorem~\ref{thm:symp_cpn_spinors} that there are no invariant spinors when $n$ is even, so it remains only to consider the case where $n$ is odd. Let $n$ be odd, and suppose that $n \geq 3$ so that $\omega^{(n \pm 3)/2} \neq 0$. Using the above formulas \eqref{eq:symp_nomizu} for the Nomizu map, we get, for $\alpha, \beta \in \C$,
			\[
			\widetilde{\Uplambda}^{a,t} \bigl( e_{4p}^{a,t} \bigr) \cdot ( \alpha \psi_+ + \beta \psi_- ) = -\sqrt{\frac{t}{2a}} \left\{ \alpha \frac{n+1}{2} y_1 \wedge y_{2p} + \beta y_{2p+1} \right\} \wedge \omega^{(n-1)/2} .
			\]
			Writing a general element $X \in \mathfrak{m}$ as a (real) linear combination
			\[ X = \mu_2 \xi_2^{a,t} + \mu_3 \xi_3^{a,t} + \sum_{p=1}^n \bigl( \mu_{4p} e_{4p}^{a,t} + \mu_{4p+1} e_{4p+1}^{a,t} + \mu_{4p+2} e_{4p+2}^{a,t} + \mu_{4p+3} e_{4p+3}^{a,t} \bigr) \]
			of the basis vectors, we find that the Clifford product with an arbitrary invariant spinor is given~by
			\begin{gather*}
				X \cdot ( \alpha \psi_+ + \beta \psi_- ) \\
				\qquad= \alpha \Bigg\{ ({\rm i} \mu_2 + \mu_3) y_1 + \sum_{p=1}^n[ ({\rm i} \mu_{4p} + \mu_{4p+1}) y_{2p} + ({\rm i} \mu_{4p+2}+\mu_{4p+3}) y_{2p+1} ]\Bigg\} \wedge \omega^{(n+1)/2} \\
				\phantom{\qquad=}{}+ \alpha \frac{n+1}{2} \Bigg\{ \sum_{p=1}^n[({\rm i} \mu_{4p} - \mu_{4p+1}) y_{2p+1} + (-{\rm i} \mu_{4p+2}+\mu_{4p+3}) y_{2p} ] \Bigg\} \wedge \omega^{(n-1)/2} \\
				\phantom{\qquad=}{}+ \beta \{ {\rm i} \mu_2 - \mu_3 \} \omega^{(n-1)/2} \\
				\phantom{\qquad=}{}+ \beta y_1 \wedge \Bigg\{ \sum_{p=1}^n[(-{\rm i} \mu_{4p} - \mu_{4p+1}) y_{2p} + (-{\rm i} \mu_{4p+2}-\mu_{4p+3}) y_{2p+1} ] \Bigg\} \wedge \omega^{(n-1)/2} \\
				\phantom{\qquad=}{}+ \beta \frac{n-1}{2} \Bigg\{ \sum_{p=1}^n[(-{\rm i} \mu_{4p} + \mu_{4p+1}) + ({\rm i} \mu_{4p+2}-\mu_{4p+3}) ] \Bigg\} y_1 \wedge \omega^{(n-3)/2} .
			\end{gather*}
			Hence, by equating coefficients in
			\begin{equation*}
				\widetilde{\Uplambda}^{a,t} \bigl( e_{4p}^{a,t} \bigr) \cdot ( \alpha \psi_+ + \beta \psi_- ) = X \cdot ( \alpha \psi_+ + \beta \psi_- ) ,
			\end{equation*}
			one easily concludes (using crucially that $n \geq 3$) that the only possibility is $\alpha = \beta = 0$, and the result then follows from the preceding discussion about the case $n=1$.
		\end{proof}
		
		We now turn to the study of the $\Sp(n+1)$-invariant spin$^{\mathbb{C}}$ spinors on $\mathbb{CP}^{2n+1}$ found in Theorem~\ref{thm:inv_gen_cpn2}. The aim is to show that, when $s' = s/2 = \pm n \pm 1$ and $t=1$, there is a pure spinor which is parallel with respect to a suitable connection on the auxiliary bundle. This encodes the fact that, for $t=1$, the metric $g_{a,t}$ is K\"{a}hler.
		
		\begin{Lemma}\label{lemma:cpn_pure_2}
			Let $k \in \mathbb{N}$ and $\psi \in \bigl\{ \omega^k \otimes \hat{1}, \bigl(y_1 \wedge \omega^k\bigr) \otimes \hat{1} , \omega^k \otimes \hat{y}_1 , \bigl(y_1 \wedge \omega^k\bigr) \otimes \hat{y}_1 \bigr\}$. Then, a~scalar multiple of $\psi$ is pure if, and only if, $k=0$ or $k=n$.
		\end{Lemma}
		\begin{proof}
			We will only prove it for $\psi = \omega^k \otimes \hat{1}$, as the other cases are analogous. Since $r=2<3$, we only need to show that \smash{$\bigl( \hat{\eta}_{12}^{\psi} \bigr)^2 = - \Id$}. Indeed, for all $1 \leq p , q \leq n$ and $\varepsilon \in \{0,1,2,3\}$, one calculates:
			\begin{gather*}
				\eta^{\psi}_{12} \bigl(\xi^{a,t}_{2}, \xi^{a,t}_{3}\bigr) = \Re \big\langle \xi^{a,t}_{2} \cdot \xi^{a,t}_{3} \cdot \hat{e}_1 \cdot \hat{e}_2 \cdot \psi , \psi \big\rangle = \Re \big\langle {\rm i} \xi^{a,t}_{2} \cdot \xi^{a,t}_{3} \cdot \psi , \psi\big\rangle \\
				 \phantom{\eta^{\psi}_{12} \bigl(\xi^{a,t}_{2}, \xi^{a,t}_{3}\bigr) }{} = - \Re \big\langle \omega^k \otimes \hat{1} , \omega^k \otimes \hat{1}\big\rangle = -(k!)^2 \binom{n}{k} , \\
				\eta^{\psi}_{12} \bigl(e^{a,t}_{4p}, e^{a,t}_{4q+1}\bigr) = \Re \big\langle e^{a,t}_{4p} \cdot e^{a,t}_{4q+1} \cdot \hat{e}_1 \cdot \hat{e}_2 \cdot \psi , \psi \big\rangle = \Re \big\langle {\rm i} \cdot e^{a,t}_{4p} \cdot e^{a,t}_{4q+1} \cdot \omega^k , \omega^k \big\rangle \\
				 \phantom{\eta^{\psi}_{12} \bigl(e^{a,t}_{4p}, e^{a,t}_{4q+1}\bigr) }{} = - \delta_{p,q} \big\langle\omega^k - 2 k y_{2p} \wedge y_{2p+1} \wedge \omega^{k-1}, \omega^k \big\rangle = -\delta_{p,q}(k!)^2 \left[ \binom{n}{k} - 2 \binom{n-1}{k-1} \right] , \\
				\eta^{\psi}_{12} \bigl(e^{a,t}_{4p+2}, e^{a,t}_{4q+3}\bigr) = -\delta_{p,q}(k!)^2 \left[ \binom{n}{k} - 2 \binom{n-1}{k-1} \right] , \\
				\eta^{\psi}_{12} \bigl(\xi^{a,t}_{2}, e^{a,t}_{4p + \varepsilon}\bigr) = \eta^{\psi}_{12} \bigl(\xi^{a,t}_{3}, e^{a,t}_{4p + \varepsilon}\bigr) = \eta^{\psi}_{12} \bigl(e^{a,t}_{4p}, e^{a,t}_{4q+2}\bigr) = \eta^{\psi}_{12} \bigl(e^{a,t}_{4p}, e^{a,t}_{4q+3}\bigr)\\
\phantom{\eta^{\psi}_{12} \bigl(\xi^{a,t}_{2}, e^{a,t}_{4p + \varepsilon}\bigr)}{} = \eta^{\psi}_{12} \bigl(e^{a,t}_{4p+1}, e^{a,t}_{4q+3}\bigr) = 0 ,
			\end{gather*}
			where $\binom{n-1}{k-1}$ is understood to be $0$ if $k = 0$. Altogether, we have
			\[
			\hat{\eta}^{\psi}_{12} = -(k!)^2 \Bigg[ \binom{n}{k} \xi^{a,t}_2 \wedge \xi^{a,t}_3 + \left[ \binom{n}{k} - 2 \binom{n-1}{k-1} \right] \sum_{p=1}^{n} \bigl( e^{a,t}_{4p} \wedge e^{a,t}_{4p+1} + e^{a,t}_{4p+2} \wedge e^{a,t}_{4p+3}\bigr) \Bigg] ,
			\]
			which is easily seen to square to a multiple of $- \Id$ if, and only if, $k=0$ or $k=n$.
		\end{proof}
		
		The preceding lemma, together with Theorem~\ref{thm:inv_gen_cpn2} describing the invariant spin$^{\C}$ spinors, implies that the $\Sp(n+1)$-invariant spin$^{\mathbb{C}}$ structure corresponding to $s = 2s' \in 2\Z$ admits invariant pure spin$^{\mathbb{C}}$ spinors if, and only if, $s' \in \{n+1,-n-1,n-1,-n+1\}$, which are given in each case by
		\begin{gather*}
			\bigl\{ ( y_1 \wedge \omega^n ) \otimes \hat{1}, 1 \otimes \hat{y}_1 \bigr\},\qquad s' = n+1, \qquad
			\bigl\{ 1 \otimes \hat{1}, (y_1 \wedge \omega^n) \otimes \hat{y}_1 \bigr\},\qquad s' = -n-1, \\
			\bigl\{ \omega^n \otimes \hat{1}, y_1 \otimes \hat{y}_1 \bigr\},\qquad s' = n-1, \qquad
			\bigl\{ y_1 \otimes \hat{1}, \omega^n \otimes \hat{y}_1 \bigr\} ,\qquad s' = -n+1.
		\end{gather*}
		
		In order to differentiate these spinors, one needs to fix a connection on the auxiliary bundle. Applying the criterion in Lemma~\ref{prelims:Nomizu_criterion}, one sees that the only $\Sp(n+1)$-invariant connection on the auxiliary bundle is the one with Nomizu map $\Uplambda \rvert_{\mathfrak{m}} = 0$. This connection, together with the Levi-Civita connection of the metric $g_{a,t}$ (with Nomizu map $\Uplambda^{a,t}$), induces a connection $\nabla^{a,t}$ on the corresponding spin$^{\mathbb{C}}$ spinor bundle. The following lemma is a straightforward calculation using the expression of the spin lift of the Nomizu map \eqref{eq:symp_nomizu}:
		
		\begin{Lemma}\label{lemma:cp2n+1_parallel}
			The invariant pure spin$^{\mathbb{C}}$ spinor $1 \otimes \hat{1}$ is $\nabla^{a,t}$-parallel if, and only if, $t=1$.
		\end{Lemma}
		
		These spin$^{\C}$ spinors encode some well-known geometric information of $\mathbb{CP}^{2n+1}$ -- see, e.g.,~\cite{ziller}.
		
		\begin{Theorem}
			The metric $g_{a,1}$ on $\mathbb{CP}^{2n+1}$ is K\"{a}hler--Einstein, for all $a > 0$.
		\end{Theorem}
		\begin{proof}
			Let $s = -2(n+1)$, and consider the spin$^{\C}$ structure on $\mathbb{CP}^{2n+1}$ determined by $\varphi_s$. By Lemmas \ref{lemma:cpn_pure_2} and \ref{lemma:cp2n+1_parallel}, this structure carries a $\nabla^{a,1}$ parallel pure spin$^{\mathbb{C}}$ spinor $\psi=1 \otimes \hat{1}$. Hence, by \cite[Corollary~4.10]{ES19}, the metric $g_{a,1}$ is K\"{a}hler, for all $a > 0$.
			
			Let us now see that these metrics are also Einstein. By Theorem~\ref{thm:herrera}\,(1), the Ricci tensor decomposes as
			\begin{equation}\label{eq:ricci_cpn_2}
				\Ric = \frac{1}{\abs{\psi}^2} \hat{\Theta}_{1 2} \circ \hat{\eta}_{1 2}^{\psi} .
			\end{equation}
			
			By calculations similar to those in the proof of Theorem~\ref{cor:cpn_conclusion}, one finds that, for $1 \leq p,q \leq n$, $0 \leq \varepsilon \leq 3$ and $2 \leq l \leq 3$,
			\begin{align*}
				&\Omega\bigl(\xi^{a,1}_2,\xi^{a,1}_3 \bigr) = - \Uplambda \bigl( \bigl[ \xi^{a,1}_2 , \xi^{a,1}_3 \bigr] \bigr) = -\frac{1}{a} \Uplambda ( \xi_1 ) = -\frac{1}{a} \bigl( \varphi_{-2(n+1)} \bigr)_* ( \xi_1 ) = \frac{2(n+1)}{a} \hat{e}_1 \wedge \hat{e}_2 , \\
				&\Omega\bigl(e^{a,1}_{4p},e^{a,1}_{4q+1} \bigr) = \Omega\bigl(e^{a,1}_{4p+2},e^{a,1}_{4q+3} \bigr) = \delta_{p,q} \frac{2(n+1)}{a} \hat{e}_1 \wedge \hat{e}_2 , \\
				&\Omega\bigl(e^{a,1}_{4p},e^{a,1}_{4q+2} \bigr) = \Omega\bigl(e^{a,1}_{4p},e^{a,1}_{4q+3} \bigr) = \Omega\bigl(e^{a,1}_{4p+1},e^{a,1}_{4q+2} \bigr) = \Omega\bigl(e^{a,1}_{4p+1},e^{a,1}_{4q+3} \bigr) = \Omega\bigl(\xi^{a,1}_l,e^{a,1}_{4p + \varepsilon} \bigr) = 0 .
			\end{align*}
			Hence, using the definition of $\hat{\Theta}_{12}$ in terms of $\Omega$ and taking $k=0$ in the proof of Lemma~\ref{lemma:cpn_pure_2},
			\begin{align*}
				&\hat{\Theta}_{1 2} = \frac{2(n+1)}{a} \Bigg( \xi^{a,1}_{2} \wedge \xi^{a,1}_{3} + \sum_{p=1}^n \bigl( {\rm e}^{a,1}_{4p} \wedge {\rm e}^{a,1}_{4p+1} + {\rm e}^{a,1}_{4p+2} \wedge {\rm e}^{a,1}_{4p+3} \bigr) \Bigg) , \\
				&\hat{\eta}_{1 2}^{\psi} = - \Bigg( \xi^{a,1}_{2} \wedge \xi^{a,1}_{3} + \sum_{p=1}^n \bigl( {\rm e}^{a,1}_{4p} \wedge {\rm e}^{a,1}_{4p+1} + {\rm e}^{a,1}_{4p+2} \wedge {\rm e}^{a,1}_{4p+3} \bigr) \Bigg) .
			\end{align*}
			Finally, substituting everything into equation \eqref{eq:ricci_cpn_2}, we get
			$
			\Ric = \frac{2(n+1)}{a} \Id$,
			which completes the proof.
		\end{proof}
		
		\subsection{Quaternionic projective space}
		
		Consider the homogeneous realisation of quaternionic projective space given by
		\[
		\mathbb{HP}^{n} \cong \faktor{\Sp(n+1)}{ \Sp(1) \times \Sp(n) } ,
		\]
		where $H \coloneqq \Sp(1) \times \Sp(n)$ is realised as a subgroup of $\Sp(n+1)$ by the upper left-hand $1 \times 1$ block for $\Sp(1)$ and the lower right-hand $n \times n$ block for $\Sp(n)$. Denote by $\mathfrak{h}$ the Lie algebra of~$H$ and $\mathfrak{h}' \coloneqq \mathfrak{sp}(n) \subseteq \mathfrak{h}$. Then,
$
			\mathfrak{h} = \mathfrak{sp}(1) \oplus \mathfrak{h}' $ (as Lie algebras),
		and explicit bases are given by
		\begin{gather*}
			\mathfrak{sp}(1) = \vecspan_{\mathbb{R}}\bigl\{ \xi_1 \coloneqq {\rm i} F_{1,1}^{(n+1)}, \xi_2 \coloneqq -{\rm k} F_{1,1}^{(n+1)}, \xi_3 \coloneqq {\rm j} F_{1,1}^{(n+1)} \bigr\} ,\\
			\mathfrak{h}' = \mathfrak{sp}(n) \\
\phantom{\mathfrak{h}' }{}= \vecspan_{\mathbb{R}} \bigl\{ {\rm i} F_{p,p}^{(n+1)}, {\rm j} F_{p,p}^{(n+1)}, {\rm k} F_{p,p}^{(n+1)}, {\rm i} F_{r,s}^{(n+1)}, {\rm j} F_{r,s}^{(n+1)}, {\rm k} F_{r,s}^{(n+1)}, E_{r,s}^{(n+1)} \bigr\}_{\begin{subarray}{l} 2 \leq r < s \leq n+1 \\ p = 2 , \dots, n+1 \end{subarray}} .
		\end{gather*}
		The isotropy subalgebra $\mathfrak{h} \subseteq \mathfrak{sp}(n+1)$ has a reductive complement
		\begin{equation*}
			\mathfrak{m} \coloneqq \vecspan_{\mathbb{R}} \bigl\{e_{4p} \coloneqq {\rm j} F_{1,p+1}^{(n+1)}, e_{4p+1} \coloneqq {\rm k} F_{1,p+1}^{(n+1)}, e_{4p+2} \coloneqq {\rm i} F_{1,p+1}^{(n+1)}, e_{4p+3} \coloneqq E_{1,p+1}^{(n+1)}\bigr\}_{p = 1, \dots, n} ,
		\end{equation*}
		and the adjoint representation of $\mathfrak{h}$ on $\mathfrak{m}$ is irreducible. Therefore, by Theorem~\ref{thm:inv_metrics}, the invariant metrics come in a one-parameter family
		$
		g_a \coloneqq a B_0 \rvert_{\mathfrak{m} \times \mathfrak{m}}$, $ a>0$,
		and one easily verifies that the above basis of $\mathfrak{m}$ rescaled by $1/\sqrt{2a}$ is $g_a$-orthonormal. Without virtually any loss of generality, in order to simplify the notation we will only consider $g:= g_{1/2}$. We take the orientation defined by the ordering $(e_4,e_5,\dots,e_{4n+3})$.
		
		\subsubsection[Invariant spin\^{}r structures]{Invariant spin$^{\boldsymbol{r}}$ structures}\label{section:hpn_structures}
		
		As $\mathbb{HP}^1$ is just the sphere $S^4$, we will suppose throughout this section that $n > 1$. By Theorem~\ref{cor:inv_gen_spin}, in order to understand $\Sp(n+1)$-invariant spin$^r$ structures on $\mathbb{HP}^n$, we need to find all Lie group homomorphisms $ \varphi \colon H \to \SO(r)$ such that $\sigma \times \varphi$ lifts to $\Spin^r(4n)$. Since $H$ is simply-connected, any such homomorphism lifts. Note also that, for $r=2$, using simplicity of~$\Sp(1)$ and $\Sp(n)$, the only Lie group homomorphism $\Sp(1) \times \Sp(n) \to \SO(2)$ is the trivial one. The corresponding $\Sp(n+1)$-invariant spin$^{\mathbb{C}}$ structure on $\mathbb{HP}^n$ is naturally induced by its unique spin structure. The first interesting case is $r=3$, which corresponds to spin$^{\mathbb{H}}$ structures. In order to classify them, we need to describe all homomorphisms $\Sp(1) \to \SO(3)$:
		
		\begin{Proposition} \label{prop:reps2}
			Up to conjugation by elements of $\SO(3)$ there are exactly two Lie group homomorphisms $\Spin(3) \to \SO(3)$, namely the trivial homomorphism and the double covering~$\lambda_3$.
		\end{Proposition}
		
		\begin{proof}
			Let $\varphi \colon \Sp(1) \to \SO(3)$ be a non-trivial homomorphism, and recall that $\Sp(1) \cong \Spin(3)$. As the Lie algebra $\mathfrak{so}(3)$ is simple, the only non-trivial normal subgroups of $\Sp(1)$ are discrete. Hence, $\varphi$ has discrete kernel. As $\Sp(1)$ is compact, the image of $\varphi$ is a closed subgroup of $\SO(3)$. By the first isomorphism theorem for Lie groups, the image of $\varphi$ is a $3$-dimensional Lie subgroup of $\SO(3)$, and hence it is open in $\SO(3)$. As $\SO(3)$ is connected, $\varphi$ must be surjective.
			
			In particular, the representation of $\Sp(1)$ on $\mathbb{R}^3$ induced by $\varphi$ must be irreducible, since otherwise the image of $\varphi$ would be contained inside a subgroup isomorphic to $\{1\}\times \SO(2)$. There is only one real irreducible $3$-dimensional representation of $\Sp(1)$ up to isomorphism~\cite[Proposition~11]{irreps} (namely, the standard spin double-cover $\varphi_0$), hence $\varphi$ is conjugate to $\varphi_0$ inside~$\GL(3,\mathbb{R})$. It follows that any two non-trivial homomorphisms $\varphi_1,\varphi_2\colon \Sp(1)\to \SO(3)$ are conjugate to each other inside $\GL(3,\mathbb{R})$, and it remains only to show that they are conjugate inside~$\SO(3)$.
			
			Fix $T \in \GL(3,\mathbb{R})$ such that, for all $A \in \Sp(1)$, we have $T^{-1} \varphi_1(A) T =\varphi_2(A)$. We claim that there exists $\hat{T} \in \SO(3)$ such that $\hat{T}^{-1} \varphi_1 \hat{T} = T^{-1} \varphi_1 T$. Indeed, let $B \coloneqq T^{-1} \varphi_1(A) T =\varphi_2(A)\in \SO(3)$. Then,
			\[
			T T^t = \varphi_1(A)^{-1} T B T^t = \varphi_1(A)^{-1} T B B^t T^t \bigl( \varphi_1(A)^{-1} \bigr)^t = \varphi_1(A)^{-1} T T^t \varphi_1(A) .
			\]
			As $\varphi_1$ is surjective, $TT^t$ commutes with all elements of $\SO(3)$, hence it is a scalar multiple of the identity. The result then follows by taking \smash{$\hat{T} = \det(T)^{-1/3} T \in \SO(3)$}.
		\end{proof}
		
		This allows us to classify invariant spin$^\mathbb{H}$ structures on quaternionic projective spaces.
		
		\begin{Theorem}
			For $n > 1$, the $\Sp(n+1)$-invariant spin$^{\mathbb{H}}$ structures on $\mathbb{HP}^{n}$ are given by
			\begin{equation*}
				\Sp(n+1) \times_{\phi_i} \Spin^{\mathbb{H}}(4n) , \qquad i = 0,1 ,
			\end{equation*}
			where $\sigma \colon H \to \SO(4n)$ is the isotropy representation, $\varphi_0$ is the trivial homomorphism $\Sp(1) \times \Sp(n) \to \SO(3)$, $\varphi_1 (x,y) = \lambda_3 (x)$ and $\phi_i$ is the unique lift of $\sigma \times \varphi_i$ to $\Spin^{\mathbb{H}}(4n)$.
		\end{Theorem}
		
		The invariant spin$^{\mathbb{H}}$ structure corresponding to $\varphi_0$ is simply the one induced by the unique spin structure, so for the rest of our discussion of $\mathbb{HP}^n$ we fix the spin$^{\mathbb{H}}$ structure corresponding to $\varphi_1$.
		
		\begin{Remark} \label{rem:hpn_aux}
			Observe that the auxiliary vector bundle of the spin$^{\mathbb{H}}$ structure corresponding to~$\varphi_1$ is $\Sp(n+1)$-equivariantly isomorphic to the rank-$3$ vector subbundle of $\End(T \mathbb{HP}^n)$ induced by the standard quaternionic K\"{a}hler structure on $\mathbb{HP}^n$.
		\end{Remark}
		
		\subsubsection[Invariant spin\^{}r spinors]{Invariant spin$^{\boldsymbol{r}}$ spinors}

		To begin, it is easy to see that this homogeneous realisation carries no invariant spinors: as the homogeneous realisation of $\mathbb{HP}^n$ that we are considering is that of a symmetric space, invariant spinors are parallel, and we know that $\mathbb{HP}^n$ cannot have any non-trivial parallel spinor, since it is not Ricci-flat. However, we shall see in the next proposition that there are always non-trivial invariant spin$^{\mathbb{H}}$ spinors, for sufficient twistings of the spinor bundle, when $n$ is odd.

		\begin{Proposition} \label{HPn_twisting_type} The $\Sp(n+1)$-invariant spinor type of $\mathbb{HP}^n$ $(n>1)$ is
			\[ \sigma(\mathbb{HP}^n,\Sp(n+1)) = \begin{cases}
				3, & n \ \text{odd}, \\
				>3 ,& n \ \text{even}.
			\end{cases}
			\]
			Furthermore, for $n$ odd, the number of twistings $m \geq 0$ of the spinor bundle which realises this is $m=n$.
		\end{Proposition}
		\begin{proof}
			The preceding discussion shows that there are no invariant spinors. As noted above, the only invariant spin$^{\C}$ structure is the one coming from the spin structure, and it is clear that there are also no invariant spin$^{\C}$ spinors in this case (since $H$ acts trivially on $\Sigma_2$ and hence each $\Sigma_{4n,2}^{m}$ is equivalent as $H$-modules to a direct sum of copies of $\Sigma_{4n}$). This shows that~${\sigma(\mathbb{HP}^n, \Sp(n+1))\geq 3}$. Denote by $V_{t}:=V(t\omega_1)$ the irreducible representation of $\mathfrak{sp}(2,\C)\cong \mathfrak{sl}(2,\C)$ with highest weight $t\omega_1$ (and dimension $t+1$). Arguing as in \cite[Section~4.1.6]{AHL}, we consider the structure of \smash{$\mathcal{S}:=(\Sigma_{4n}\rvert_{\mathfrak{h}^{\C}})^{\mathfrak{sp}(2n,\C)} = \text{span}_{\C}\{ \omega^k\}_{k=0}^n$} (\smash{$\omega:=\sum_{j=1}^n y_{2j}\wedge y_{2j+1} $}) as a module for~${\mathfrak{sp}(2,\C)\subset \mathfrak{h}^{\C}}$ (here we adopt the usual convention $\mathfrak{sp}(k)^{\C} \cong \mathfrak{sp}(2k,\C)$). The action of~$\mathfrak{sp}(2,\C)$ on $\mathcal{S}$ follows from \cite[Lemma~5.13]{jordansasakian} and is given explicitly by
			\begin{align}
				&\widetilde{\ad}(\xi_1)\restr{\mathfrak{m}} \cdot \omega^k = {\rm i}(n-2k) \omega^k ,\label{CSA_action_quaternionic_case} \\
				&\widetilde{\ad}(\xi_2)\restr{\mathfrak{m}} \cdot \omega^k = k (n-k+1) \omega^{k-1} - \omega^{k+1} , \nonumber\\
				&\widetilde{\ad}(\xi_3)\restr{\mathfrak{m}} \cdot \omega^k = {\rm i} \bigl( \omega^{k+1} + k (n-k+1) \omega^{k-1} \bigr) .\label{xi3_action_quaternionic_case}
			\end{align}
			The standard basis element for the Cartan subalgebra of $\mathfrak{sp}(2,\C)\cong \mathfrak{sl}(2,\C)$ is $-{\rm i} \xi_1\sim \text{diag}[1,-1]$, and by \eqref{CSA_action_quaternionic_case} we see that the action of this element on $\mathcal{S}$ has highest eigenvalue $n$. In particular, $\mathcal{S}$ contains a copy of $V_n$, and by reason of dimension we have $\mathcal{S}\simeq V_n$ as $\mathfrak{sp}(2,\C)$-modules. Note that this representation is self-dual (see, e.g., \cite{Samelson}). Thus, it suffices to show that the smallest odd tensor power of $\Sigma_3$ which contains a copy of $V_n$ is $\Sigma_3^{\otimes n}$. Recalling the well-known decomposition
			\begin{align}\label{sl2_reps_tensor_decomp}
				V_s\otimes V_t \simeq V_{s+t}\oplus V_{s+t-2}\oplus \dots \oplus V_{s-t}, \qquad s\geq t
			\end{align}
			of $\mathfrak{sl}(2,\C)$-representations (see, e.g., \cite[Sections~11.1 and 11.2] {FultonHarris}), the result for the case where~$n$ is odd follows by repeatedly using \eqref{sl2_reps_tensor_decomp} to decompose tensor powers of $\Sigma_3 \simeq V_1$. For the case where $n$ is even, one sees from \eqref{sl2_reps_tensor_decomp} that the decompositions of odd tensor powers of $\Sigma_3\simeq V_1$ into irreducible representations contain only factors of the form $V_t$ with $t$ odd, and in particular cannot contain a copy of $V_n$.
		\end{proof}

		\begin{Remark}
			The difference in behaviour between the even and odd cases in the preceding proposition occurs as something of a technicality rather than a manifestation of any significant geometric difference; indeed, the argument presented in the odd case also produces representation-theoretic invariants in the even case if we allow even twistings of the spinor bundle. The reason to exclude even twistings is that the twisted spinor module would then fail to be well defined as a representation of $\Spin^r(n)$, since $[-1,-1]$ wouldn't act by the identity map.\ In order to obtain a notion of spinor bundles with even numbers of twistings, one needs to consider instead the alternative structure groups $\Spin(n) \times \Spin(r)$ described in \cite[Remark~2.3]{ES19}.
		\end{Remark}
		
		In the next proposition we describe explicitly the invariant $n$-twisted spin$^{\mathbb{H}}$ spinors which realise the equality $\sigma(\mathbb{HP}^n,\Sp(n+1))=3$ for $n$ odd. Recall that, as noted in Remark~\ref{rem:hpn_aux}, the auxiliary vector bundle $E$ of the spin$^\mathbb{H}$ structure corresponding to $\varphi_1$ can be seen as a~subbundle of the endomorphism bundle $\End(T \mathbb{HP}^n)$. In particular, $E$ inherits a natural connection $\nabla^E$ from the Levi-Civita connection $\nabla^{g}$ on $T \mathbb{HP}^n$, and the former induces a connection~${\nabla^{g,E}:=\nabla^g\otimes\bigl(\nabla^E\bigr)^{\otimes m}}$ on the spinor bundle $\Sigma_{4n,3}^m \mathbb{HP}^n$ for any odd $m \geq 1$. Recall that~${\{\xi_1,\xi_2,\xi_3\}}$ is a basis of $\mathfrak{sp}(1) \subset \mathfrak{sp}(1) \oplus \mathfrak{sp}(n) = \mathfrak{h}$, and denote by $ ( \Phi_1 \coloneqq \ad(\xi_1) \restr{\mathfrak{m}} , \allowbreak \Phi_2 \coloneqq \ad(\xi_2) \restr{\mathfrak{m}} , \Phi_3 \coloneqq \ad(\xi_3) \restr{\mathfrak{m}} )$ the standard basis of the invariant rank-$3$ subspace of $\End(\mathfrak{m})$ corresponding to $E$. The action of $\xi_i$ on this subspace is given by
		\begin{align*}
			\Phi_i&\mapsto 0, \qquad \Phi_j\mapsto 2\Phi_k, \qquad \Phi_k\mapsto -2\Phi_j ,
		\end{align*}
		where $(i,j,k)$ is an even permutation of $(1,2,3)$,
		and the spin lift of this representation acts on~${\Sigma_3\cong \C^2}$ by the standard basis matrices for $\mathfrak{su}(2)$:
		\begin{align*}
			\rho(\xi_1):= \begin{pmatrix}
				{\rm i} & 0\\ 0 & -{\rm i}
			\end{pmatrix} , \qquad \rho(\xi_2):= \begin{pmatrix}
				0 & 1\\ -1 & 0
			\end{pmatrix}, \qquad
			\rho(\xi_3):= \begin{pmatrix}
				0 & {\rm i} \\ {\rm i} & 0
			\end{pmatrix}
		\end{align*}
		(these matrices are taken relative to the standard basis $\hat{1}:=(1,0)$, $\hat{y}_1:= (0,1)$ for $\Sigma_3\cong \C^2 $). In order to relate these to the usual presentation of the Lie algebra $\mathfrak{sl}(2,\C)\cong \mathfrak{sp}(2,\C)$, we introduce~${H:=-{\rm i}\xi_1}$, $X:=\frac{1}{2}(\xi_2-{\rm i} \xi_3)$, $Y:= - \frac{1}{2}(\xi_2+{\rm i}\xi_3) $, so that
		\begin{align}\label{sp1repstandardmatrices}
			\rho(H)=\begin{pmatrix}
				1&0 \\ 0&-1
			\end{pmatrix},\qquad 	\rho(X)= \begin{pmatrix}
				0& 1\\ 0& 0
			\end{pmatrix} , \qquad \rho(Y)= \begin{pmatrix}
				0&0\\ 1&0
			\end{pmatrix}
		\end{align}
		act in the representation $\Sigma_3\cong \C^2$ by the usual operators. Using this setup, we obtain the following.

		\begin{Theorem}\label{HPn_twisted_inv_spinor_basis} If $n>1$ is odd, the space of $\Sp(n+1)$-invariant $n$-twisted spin$^{\mathbb{H}}$ spinors is spanned over $\C$ by \[\psi:= \sum_{j=0}^{n} (-1)^{j} \omega^j \otimes \bigl(\rho(Y)^{n-j}.\mathds{1}\bigr) , \]
			where $\mathds{1}:=\hat{1} \otimes \overset{n}{\cdots} \otimes \hat{1}$.
		\end{Theorem}
		\begin{proof}
			As in the proof of Proposition~\ref{HPn_twisting_type}, we note that \smash{$\mathcal{S}:=(\Sigma_{4n}\rvert_{\mathfrak{h}^{\C}})^{\mathfrak{sp}(2n,\C)}\simeq V_n$} as modules for $ \mathfrak{sl}(2,\C)\cong \mathfrak{sp}(2,\C)\subset \mathfrak{h}^{\C}$; explicitly we have \smash{$\mathcal{S} = \Span_{\C} \bigl\{ \omega^{\ell}\bigr\}_{\ell=0}^n$}, with the action of $\mathfrak{sp}(2,\C)=\Span_{\C}\{\xi_1,\xi_2,\xi_3\}$ by the formulas \eqref{CSA_action_quaternionic_case}--\eqref{xi3_action_quaternionic_case}. It is clear from \eqref{sp1repstandardmatrices} that $\mathds{1} \in \Sigma_3^{\otimes n}$ is a~highest weight vector for $\mathfrak{sp}(2,\C)$, and that the $\mathfrak{sp}(2,\C)$-submodule $ U(\mathfrak{sp}(2,\C)).\mathds{1}$ that it generates\footnote{Here, $U(\mathfrak{sp}(2,\C))$ refers to the universal enveloping algebra of $\mathfrak{sp}(2,\C)$ and the $.$ product refers to the action via the representation.} is isomorphic to $V_n$ (since $\rho(H) = \text{diag}[1,-1]$ acts on $\mathds{1}$ by multiplication by $n$). Therefore, we have \smash{$U(\mathfrak{sp}(2,\C)).\mathds{1}= \Span_{\C} \bigl\{\rho(Y)^k.\mathds{1}\bigr\}_{k=0}^n$}. On the other hand, we see from \eqref{CSA_action_quaternionic_case}--\eqref{xi3_action_quaternionic_case} that~${1=\omega^0\in \mathcal{S}}$ is a highest weight vector and $Y.\omega^k= \omega^{k+1}$ for all $k=0,\dots,n$. In particular, the isomorphism $\mathcal{T}\colon\mathcal{S} \to U(\mathfrak{sp}(2,\C)).\mathds{1} $ is given (up to rescaling) by $ \mathcal{T}\bigl(\omega^k\bigr) = \rho(Y)^k.\mathds{1} $. We have \smash{$\mathcal{T}\in \Hom_{\mathfrak{sp}(2,\C)}\bigl(\mathcal{S}, \Sigma_3^{\otimes n}\bigr) \simeq \mathcal{S}^*\otimes \Sigma_3^{\otimes n}$}, and the invariant spin$^{\mathbb{H}}$ spinor we are seeking is the corresponding element of $\mathcal{S}\otimes \Sigma_3^{\otimes n}$ obtained via the musical isomorphism $\sharp\colon \mathcal{S}^*\simeq \mathcal{S}$ associated to the $\mathfrak{sp}(2,\C)$-invariant symplectic form
			\[ \Omega\bigl(\omega^j,\omega^k\bigr) = \begin{cases}
				(-1)^j, & j+k=n , \\ 0, & j+k\neq n
			\end{cases} \]
			on $\mathcal{S}$. Defining \smash{$\widehat{\omega^j}\in \mathcal{S}^*$} by \smash{$\widehat{\omega^j}\bigl(\omega^k\bigr)=\delta_{j,k}$}, one sees that \smash{$\bigl(\widehat{\omega^j}\bigr)^{\sharp}= 	(-1)^{j+1}\omega^{n-j} $}, and the result~then follows by noting that \smash{$\mathcal{T} = \sum_{j=0}^{n} \widehat{\omega^j} \otimes \bigl(\rho(Y)^j.\mathds{1}\bigr)$}.
		\end{proof}
		
\begin{Remark}
The spinor in the statement of Theorem~\ref{HPn_twisted_inv_spinor_basis} corresponds to the one in \cite[Section~3.4.1]{ES19}, which the authors show to be pure.
\end{Remark}
		
		Finally, we give the differential equation satisfied by $\psi$. Recall that $\mathbb{HP}^n=\Sp(n+1)/\Sp(1)\times \Sp(n)$ is a symmetric space, and that the auxiliary bundle of the spin$^{\mathbb{H}}$ structure under consideration is isomorphic to the rank-$3$ bundle spanned by the (locally-defined) endomorphisms~$\Phi_1$, $\Phi_2$, $\Phi_3$. This bundle inherits a connection $\nabla^E$ from the Levi-Civita connection, and its Nomizu map vanishes identically when restricted to $\mathfrak{m}$. The following is an immediate consequence of Proposition~\ref{symmetricspace_differential_equation}.

		\begin{Corollary}
			For $n > 1$ odd, the invariant $n$-twisted spin$^{\mathbb{H}}$ spinor $\psi$ in Theorem~{\rm\ref{HPn_twisted_inv_spinor_basis}} is parallel with respect to the invariant connection $\nabla^{g,E}:=\nabla^g \otimes\bigl(\nabla^E\bigr)^{\otimes n}$.
		\end{Corollary}
		This spin$^{\mathbb{H}}$ spinor encodes, via \cite[Corollary~4.12]{ES19}, a well-known geometric fact -- see, e.g.,~\cite{besse}.

		\begin{Theorem}
			The metric $g_a$ on $\mathbb{HP}^n$ is quaternionic K\"{a}hler.
		\end{Theorem}
		
		\subsection{Octonionic projective plane}\label{section:op2}
		
		Consider now the octonionic projective plane, realised as a homogeneous space via
		\[
		\mathbb{OP}^{2} \cong \faktor{F_4}{\Spin(9)} .
		\]
		A description of the isometric action of $F_4$ can be found, e.g., in \cite{Baez_octonions}, and, importantly, the isotropy representation is just the real spin representation of $\Spin(9)$ on $\mathbb{R}^{16}$: %-- note that $9 = 2 \cdot 4 + 1$, and $16 = 2^4$
		\[
		\Spin(9) \xhookrightarrow{} \Cl^{0}_{9,0} \cong \Cl_{8,0} \cong \Mat_{16}( \mathbb{R} ) .
		\]
		As in \cite[Theorems~1.4.3 and 1.4.4, Proposition~1.4.5]{wernli}, at the level of Lie algebras this inclusion is given by
		\begin{align}
				\mathfrak{spin}(9) \xhookrightarrow{} \Cl^{0}_{9,0} &\overset{\psi_1}{\cong} \Cl_{8,0} \overset{\psi_2}{\cong} \Cl_{0,6} \otimes \Cl_{2,0} \overset{\psi_3}{\cong} \Cl_{4,0} \otimes \Cl_{0,2} \otimes \Cl_{2,0} \nonumber\\
				&\overset{\psi_4}{\cong} \Cl_{0,2} \otimes \Cl_{2,0} \otimes \Cl_{0,2} \otimes \Cl_{2,0} \overset{\psi_5}{\cong} \Mat_2 ( \mathbb{R} ) \otimes \mathbb{H} \otimes \Mat_2 ( \mathbb{R} ) \otimes \mathbb{H}\nonumber \\
				&\overset{\psi_6}{\cong} \Mat_2 ( \mathbb{R} ) \otimes \Mat_2 ( \mathbb{R} ) \otimes \mathbb{H} \otimes \mathbb{H} \overset{\psi_7}{\cong} \Mat_{4} ( \mathbb{R} ) \otimes \Mat_{4} ( \mathbb{R} ) \overset{\psi_8}{\cong} \Mat_{16} ( \mathbb{R} ) . \label{eq:chain}
		\end{align}
		The algebra isomorphisms $\psi_{1}$, $\psi_{2}$, $\psi_{3}$, $\psi_{4}$, $\psi_{5}$ are given explicitly in \cite{wernli}; $\psi_{6}$ is the obvious permutation of the second and third factors; $\psi_{7}$ is the tensor product of the Kronecker product of the first two factors and the isomorphism
		\begin{align*}
				\mathbb{H} \otimes_{\mathbb{R}} \mathbb{H} &\to \Mat_{4}( \mathbb{R} ),\qquad
				q_1 \otimes q_2 \mapsto ( x \mapsto q_1 \cdot x \cdot \overline{q_2} ),
		\end{align*}
		where $x = ( x_1 , x_2, x_3, x_4 ) \in \mathbb{R}^4$ is thought of as the quaternion $x_1 + {\rm i} x_2 + {\rm j} x_3 + {\rm k} x_4$; and $\psi_{8}$ is the Kronecker product. Letting $ \{e_0, e_1 , \dots , e_8 \}$ be the canonical basis of $\mathbb{R}^9$, a basis of $\mathfrak{spin}(9)$ is given by
		\[
		\mathfrak{spin}(9)= \Span_{\R} \{ e_i \cdot e_j \}_{0 \leq i < j \leq 8} .
		\]
		With a slight abuse of notation, we will denote by $e_1, \dots, e_n$ the elements of the canonical basis of $\mathbb{R}^n$, for $n=2,4,6,8, 16$. Now we give the images in $\Mat_{16}( \mathbb{R} )$ of each of the elements of our basis of $\mathfrak{spin}(9)$. For the first basis vector $e_0\cdot e_1$, one computes, following the chain of maps in~\eqref{eq:chain}
		\begin{align}
			e_0 \cdot e_1 &\mapsto e_0 \cdot e_1 \mapsto e_1 \mapsto e_1 \otimes ( e_1 \cdot e_2 ) \mapsto e_1 \otimes ( e_1 \cdot e_2 ) \otimes ( e_1 \cdot e_2 ) \nonumber\\
			& \mapsto e_1 \otimes ( e_1 \cdot e_2 ) \otimes ( e_1 \cdot e_2 ) \otimes ( e_1 \cdot e_2 ) \mapsto \begin{pmatrix} 1 & 0 \\ 0 & -1 \end{pmatrix} \otimes k \otimes \begin{pmatrix} 0 & 1 \\ -1 & 0 \end{pmatrix} \otimes k \nonumber\\
			& \mapsto \begin{pmatrix} 1 & 0 \\ 0 & -1 \end{pmatrix} \otimes \begin{pmatrix} 0 & 1 \\ -1 & 0 \end{pmatrix} \otimes k \otimes k \mapsto \begin{pmatrix} 0 & 1 & 0 & 0 \\ -1& 0 & 0 & 0 \\ 0 & 0 & 0 & -1 \\ 0 & 0 & 1 & 0 \end{pmatrix} \otimes \begin{pmatrix} 1 & 0 & 0 & 0 \\ 0 & -1 & 0 & 0 \\ 0 & 0 & -1 & 0 \\ 0 & 0 & 0 & 1 \end{pmatrix} \nonumber\\
			&\mapsto -E^{(16)}_{1,5} + E^{(16)}_{2,6} + E^{(16)}_{3,7} - E^{(16)}_{4,8} + E^{(16)}_{9,13} - E^{(16)}_{10,14} - E^{(16)}_{11,15} + E^{(16)}_{12,16} .
			\label{eqn_number_spin9_parti}\end{align}
		The others are computed similarly, giving
		\begin{gather}
				e_0 \cdot e_2 \mapsto -E^{(16)}_{1,13} + E^{(16)}_{2,14} + E^{(16)}_{3,15} - E^{(16)}_{4,16} + E^{(16)}_{5,9} - E^{(16)}_{6,10} - E^{(16)}_{7,11} + E^{(16)}_{8,12} ,\nonumber
				\\
				e_0 \cdot e_3
				\mapsto -E^{(16)}_{1,7} - E^{(16)}_{2,8} - E^{(16)}_{3,5} - E^{(16)}_{4,6} - E^{(16)}_{9,15} - E^{(16)}_{10,16} - E^{(16)}_{11,13} - E^{(16)}_{12,14} ,\nonumber
				\\
				e_0 \cdot e_4
				\mapsto E^{(16)}_{1,6} + E^{(16)}_{2,5} - E^{(16)}_{3,8} - E^{(16)}_{4,7} + E^{(16)}_{9,14} + E^{(16)}_{10,13} - E^{(16)}_{11,16} - E^{(16)}_{12,15} ,\nonumber
				\\
				e_0 \cdot e_5
				\mapsto -E^{(16)}_{1,4} + E^{(16)}_{2,3} + E^{(16)}_{5,8} - E^{(16)}_{6,7} - E^{(16)}_{9,12} + E^{(16)}_{10,11} + E^{(16)}_{13,16} - E^{(16)}_{14,15} ,\nonumber
				\\
				e_0 \cdot e_6
				\mapsto -E^{(16)}_{1,8} + E^{(16)}_{2,7} - E^{(16)}_{3,6} + E^{(16)}_{4,5} - E^{(16)}_{9,16} + E^{(16)}_{10,15} - E^{(16)}_{11,14} + E^{(16)}_{12,13} ,\nonumber
				\\
				e_0 \cdot e_7
				\mapsto -E^{(16)}_{1,2} + E^{(16)}_{3,4} - E^{(16)}_{5,6} + E^{(16)}_{7,8} - E^{(16)}_{9,10} + E^{(16)}_{11,12} - E^{(16)}_{13,14} + E^{(16)}_{15,16} ,\nonumber
				\\
				e_0 \cdot e_8
				\mapsto -E^{(16)}_{1,3} - E^{(16)}_{2,4} - E^{(16)}_{5,7} - E^{(16)}_{6,8} - E^{(16)}_{9,11} - E^{(16)}_{10,12} - E^{(16)}_{13,15} - E^{(16)}_{14,16} .
				\label{eqn_number_spin9_partii}
		\end{gather}
		The images of the other basis vectors $e_i\cdot e_j$ ($1\leq i\leq 8$) for $\mathfrak{spin}(9)$ are then determined by taking products of the above generators inside $\Cl^{0}_{9,0}$ using the Clifford algebra identities $e_i\cdot e_j = (e_0\cdot e_i)\cdot (e_0\cdot e_j)$.
		
		\subsubsection[Invariant spin\^{}r structures]{Invariant spin$^{\boldsymbol{r}}$ structures}
		
		By Theorem~\ref{cor:inv_gen_spin}, in order to understand $F_4$-invariant spin$^r$ structures on $\mathbb{OP}^2$, we need to find all Lie group homomorphisms $ \varphi \colon \Spin(9) \to \SO(r)$ such that $\sigma \times \varphi$ lifts to $\Spin^r(16)$ (of course, as the group $\Spin(9)$ is simply connected, the lifting condition is automatically satisfied). As the Lie algebra $\mathfrak{spin}(9) \cong \mathfrak{so}(9)$ is simple, for $1 \leq r \leq 8$ the only homomorphism $\Spin(9)\to \SO(r)$ is the trivial one. The corresponding $F_4$-invariant spin$^r$ structures are just the ones in the lineage of the invariant spin structure.
		
		The first non-trivial case is $r=9$, where we have the covering homomorphism $\lambda_9 \colon \Spin(9) \to \SO(9)$. By essentially the same argument as in Proposition~\ref{prop:reps2}, there are only two Lie group homomorphisms $\Spin(9) \to \SO(9)$ up to conjugation by elements of $\SO(9)$, namely the trivial one $\varphi_0$ and the double covering $\varphi_1 \coloneqq \lambda_9$, leading to two possible invariant spin$^9$ structures up to equivalence.
		
		\begin{Theorem}\label{F4_inv_gen_spin}
			The $F_4$-invariant spin$^{9}$ structures on $\mathbb{OP}^{2}$ are given by
		$
				F_4 \times_{\phi_i} \Spin^{9}(16)$, ${i = 0,1}$,
			where $\phi_i$ is the unique lift of $\sigma \times \varphi_i$ to $\Spin^{9}(16)$ and $\sigma \colon \Spin(9) \to \SO(16)$ is the isotropy representation.
		\end{Theorem}
		
		\subsubsection[Invariant spin\^{}r spinors]{Invariant spin$^{\boldsymbol{r}}$ spinors}
		
		As in the case of $\mathbb{HP}^n$, the octonionic projective plane $\mathbb{OP}^2$ does not admit any invariant spinors (cf.\ the discussion before Proposition~\ref{HPn_twisting_type}), so from this point forward we consider the non-trivial invariant spin$^9$ structure (i.e., the one corresponding to $i=1$ in the preceding theorem). In order to describe its twisted spin$^9$ spinors we first need a small lemma describing the decomposition of the spin lift of the isotropy representation as a direct sum of highest weight modules. This result may be found, written in a slightly different form and without proof, in \cite[Section~7]{Fri01_weak_spin9}; we include a sketch of the proof here as the notation and formulas will be useful for subsequent discussion.

		\begin{Lemma}[{\cite{Fri01_weak_spin9}}]\label{Sigma16_decomp_lemma}
			 As modules for $\Spin(9)^{\C}$, the spin lift $\widetilde{\sigma}$ of the isotropy representation decomposes as
			\begin{align}\label{Sigma16_decomp}
				\Sigma_{16} \simeq \underbrace{V(\omega_1+\omega_4)}_{\Sigma_{16}^-} \oplus \underbrace{V(\omega_3) \oplus V(2\omega_1)}_{\Sigma_{16}^+},
			\end{align}
			where $V(\mu)$ denotes the irreducible representation of highest weight $\mu$ and $\omega_i$, $i=1,2,3,4$ denote the fundamental weights of $\mathfrak{spin}(9)^{\C} \cong \mathfrak{so}(9,\C)$.
		\end{Lemma}
		\begin{proof}
			In order to take advantage of the explicit operators calculated in \eqref{eqn_number_spin9_parti}--\eqref{eqn_number_spin9_partii}, we view~$ { \mathfrak{spin}(9)^{\C}\cong \mathfrak{so}(9,\C)}$ as the set of $9\times 9$ skew-symmetric matrices in $\mathfrak{gl}(9,\C)$. We take the (real form of the) Cartan subalgebra spanned by \smash{$h_j:= -{\rm i}E^{(9)}_{2j-1,2j}$}, $j=1,2,3,4$, together with the (positive) re-scaling of the Killing form such that the $h_j$'s are orthogonal and unit length. Letting $v_j:=h_j^{\flat}$, we have the simple roots \[\alpha_1=v_1-v_2, \qquad \alpha_2=v_2-v_3, \qquad \alpha_3=v_3-v_4,\qquad \alpha_4=v_4, \]
			and the corresponding fundamental weights \smash{$\omega_j:=\frac{2\alpha_j}{||\alpha_j||^2}$} are given by
			\[
			\omega_1=v_1,\qquad \omega_2=v_1+v_2,\qquad \omega_3=v_1+v_2+v_3, \qquad \omega_4 = \frac{1}{2}(v_1+v_2+v_3+v_4).\]
			The root vectors $X_i:= X_{\alpha_i}$ associated to the simple roots $\alpha_i$ are
			\begin{gather*}
				X_{1}= E^{(9)}_{1, 3} + E^{(9)}_{2, 4} +
				{\rm i}\bigl(-E^{(9)}_{2, 3} + E^{(9)}_{1, 4}\bigr),\qquad
				X_{2} = E^{(9)}_{3, 5} + E^{(9)}_{4, 6} +
				{\rm i}\bigl(-E^{(9)}_{4, 5} + E^{(9)}_{3, 6}\bigr),\\
				X_{3} = E^{(9)}_{5, 7} + E^{(9)}_{6, 8} +
				{\rm i}\bigl(-E^{(9)}_{6, 7} + E^{(9)}_{5, 8}\bigr),\qquad
				X_{4} = E^{(9)}_{7, 9} - {\rm i}E^{(9)}_{8, 9},
			\end{gather*}
and the root vectors associated to the roots $-\alpha_i$, $i=1,2,3,4$ are given by $Y_i:=Y_{\alpha_i}:=\overline{X_{i}}$. We note that this setup is slightly unusual\footnote{One usually chooses a different realization of the Lie algebra $\mathfrak{so}(9,\C)$ in order to make the elements of the Cartan subalgebra diagonal matrices, but that realization is less convenient for our purposes here.} and can be found, e.g., in \cite{Ziller_2010}. Using the explicit formulas for $\sigma$ from \eqref{eqn_number_spin9_parti}--\eqref{eqn_number_spin9_partii}, we find that the Cartan subalgebra generators $h_i$ and simple root vectors $X_i$ act in the complexified isotropy representation $\mathfrak{m}^{\C}\simeq \Sigma_9$ by the operators
			\begin{align*}
				\sigma(h_1) &= -\frac{{\rm i}}{2}( -e_{1, 5} + e_{2, 6} + e_{3, 7} -
				e_{4, 8} + e_{9, 13} - e_{10, 14} - e_{11, 15} +
				e_{12, 16}) , \\
				\sigma(h_2) &= -\frac{{\rm i}}{2}( e_{1, 11} - e_{2, 12} - e_{3, 9} +
				e_{4, 10} + e_{5, 15} - e_{6, 16} - e_{7, 13} +
				e_{8, 14}) , \\
				\sigma(h_3) &=-\frac{{\rm i}}{2}( e_{1, 7} - e_{2, 8} - e_{3, 5} + e_{4, 6} +
				e_{9, 15} - e_{10, 16} - e_{11, 13} + e_{12, 14} ), \\
				\sigma(h_4) &=-\frac{{\rm i}}{2}( -e_{1, 7} - e_{2, 8} + e_{3, 5} +
				e_{4, 6} - e_{9, 15} - e_{10, 16} + e_{11, 13} +
				e_{12, 14}),
			\end{align*}
			and
			\begin{gather*}
				2\sigma(X_{1})= e_{1, 3} - {\rm i} e_{1, 7} - {\rm i} e_{1, 9} -
				e_{1, 13} - e_{2, 4} - {\rm i} e_{2, 8} -
				{\rm i} e_{2, 10} + e_{2, 14} - {\rm i} e_{3, 5} -
				{\rm i} e_{3, 11} \\
\phantom{2\sigma(X_{1})= }{} + e_{3, 15} - {\rm i} e_{4, 6} -
				{\rm i} e_{4, 12} - e_{4, 16} + e_{5, 7} +
				e_{5, 9} - {\rm i} e_{5, 13} - e_{6, 8} -
				e_{6, 10} - {\rm i} e_{6, 14} \\
\phantom{2\sigma(X_{1})= }{} - e_{7, 11} -
				{\rm i} e_{7, 15} + e_{8, 12} - {\rm i} e_{8, 16} -
				e_{9, 11} - {\rm i} e_{9, 15} + e_{10, 12} -
				{\rm i} e_{10, 16} - {\rm i} e_{11, 13} \\
\phantom{2\sigma(X_{1})= }{} - {\rm i} e_{12, 14} -
				e_{13, 15} + e_{14, 16} , \\
				2\sigma(X_{2}) = -{\rm i} e_{1, 4} + e_{1, 6} - e_{1, 10} +
				{\rm i} e_{1, 16} - {\rm i} e_{2, 3} - e_{2, 5} +
				e_{2, 9} + {\rm i} e_{2, 15} + e_{3, 8} -
				e_{3, 12} \\
\phantom{2\sigma(X_{2}) = }{}- {\rm i} e_{3, 14} - e_{4, 7} +
				e_{4, 11} - {\rm i} e_{4, 13} - {\rm i} e_{5, 8} +
				{\rm i} e_{5, 12} - e_{5, 14} - {\rm i} e_{6, 7} +
				{\rm i} e_{6, 11} + e_{6, 13} \\
\phantom{2\sigma(X_{2}) = }{}- {\rm i} e_{7, 10} -
				e_{7, 16} - {\rm i} e_{8, 9} + e_{8, 15} -
				{\rm i} e_{9, 12} + e_{9, 14} - {\rm i} e_{10, 11} -
				e_{10, 13} + e_{11, 16} \\
\phantom{2\sigma(X_{2}) = }{} - e_{12, 15} -
				{\rm i} e_{13, 16} - {\rm i} e_{14, 15} ,\\
				2\sigma(X_{3}) = -2 e_{1, 3} - 2 {\rm i} e_{1, 5} - 2 {\rm i} e_{3, 7} +
				2 e_{5, 7} - 2 e_{9, 11} - 2 {\rm i} e_{9, 13} -
				2 {\rm i} e_{11, 15} + 2 e_{13, 15} ,\\
				2\sigma(X_{4})= {\rm i} e_{1, 4} + e_{1, 6} - {\rm i} e_{2, 3} -
				e_{2, 5} - e_{3, 8} + e_{4, 7} +
				{\rm i} e_{5, 8} - {\rm i} e_{6, 7} + {\rm i} e_{9, 12} +
				e_{9, 14} \\
\phantom{2\sigma(X_{4})=}{}- {\rm i} e_{10, 11} - e_{10, 13} -
				e_{11, 16} + e_{12, 15} + {\rm i} e_{13, 16} -
				{\rm i} e_{14, 15} .
			\end{gather*}
			Considering the action of the lifts $\widetilde{\sigma}(h_i)$, $\widetilde{\sigma}(X_{i}) \in\mathfrak{spin}(16)^{\C} $ in the spin representation, a straightforward calculation using computer algebra software yields three linearly independent joint eigenvectors for the $\widetilde{\sigma}(h_i)$ which are simultaneously annihilated by the action of each $\widetilde{\sigma}(X_{i})$ (i.e., highest weight vectors). The corresponding weights are
			\[
			\frac{1}{2}(3v_1 +v_2+v_3+v_4) = \omega_1 + \omega_4, \qquad v_1+v_2+v_3 = \omega_3 , \qquad 2v_1 = 2\omega_1 ,
			\]
			and the assertion that $\Sigma_{16}^- \simeq V(\omega_1+\omega_4)$ and $\Sigma_{16}^+ \simeq V(\omega_3)\oplus V(2\omega_1) $ may be deduced from \cite[Section~7]{Fri01_weak_spin9}.
		\end{proof}

		Note that the preceding lemma immediately recovers the fact that the invariant spin structure carries no invariant spinors. It also allows us to readily describe the smallest twisting for which~$\mathbb{OP}^2$ admits invariant twisted spin$^9$ spinors.
		\begin{Theorem}
			The $F_4$-invariant spinor type of $\mathbb{OP}^2$ is $\sigma\bigl(\mathbb{OP}^2,F_4\bigr)=9$, and the twisting of the spinor bundle which realises this is $m=3$. Furthermore, the space of invariant $3$-twisted spin$^9$ spinors has complex dimension $4$.
		\end{Theorem}
		\begin{proof}
			First, we recall that every representation of $\mathfrak{so}(9,\C)$ is self-dual (see, e.g., \cite{Samelson}). Therefore, using a similar argument as in the proof of Proposition~\ref{HPn_twisting_type}, and in light of the preceding lemma, it suffices to show that $\Sigma_9^{\otimes 3}$ is the smallest odd tensor power of $\Sigma_9\simeq V(\omega_4)$ which contains a copy of $V(\omega_1+\omega_4)$, $V(\omega_3)$, or $V(2\omega_1)$. It is easily verified using, e.g., the LiE software package \cite{LiE} that \begin{align} \label{Sigma93_decomp}\Sigma_9^{\otimes 3} &\simeq 5 V(\omega_4) \oplus V(3\omega_4) \oplus 2V(\omega_3+\omega_4)\oplus 3V(\omega_2+\omega_4) \oplus 4V(\omega_1+\omega_4 ) .
			\end{align}
			Finally, using self-duality, it follows from (\ref{Sigma16_decomp}) and (\ref{Sigma93_decomp}) that
			\[
			\dim_{\C} \bigl(\Sigma_{16,9}^3\bigr)_{\inv} = \dim_{\C} \bigl(\Sigma_{16} \otimes \Sigma_9^{\otimes 3}\bigr)^{\Spin(9)} = \dim_{\C} \text{Hom}_{\Spin(9)}\bigl(\Sigma_{16} , \Sigma_9^{\otimes 3}\bigr) = 4.\tag*{\qed}
			\] \renewcommand{\qed}{}
		\end{proof}

		Now we examine more closely the invariant 3-twisted spin$^9$ spinors from the preceding theorem. This $4$-dimensional space corresponds to the pairings of the $4$ copies of $V(\omega_1+\omega_4)$ in~(\ref{Sigma93_decomp}) with the single copy in~(\ref{Sigma16_decomp}), so in order to obtain formulas for the spinors we first need to clarify the algebraic structure of this representation. With all notation as above, one finds using computer algebra software an explicit highest weight vector $w_0$ (unique up to scaling) generating $\Sigma_{16}^- \simeq V(\omega_1+\omega_4)\subseteq \Sigma_{16}$, and one may verify furthermore that any other weight vector can be obtained from $w_0$ by applying at most $18$ lowering operators $Y_i$, $i=1,2,3,4$. Writing $Y_{\mathcal{I}} := Y_{i_1}.Y_{i_2}\dots Y_{i_k}$ for a multi-index $\mathcal{I}=\{i_1,\dots,i_k\}$, one possible minimal choice of multi-indices $\bigcup_{k=0}^{18}\{ \mathcal{I}_{k,\ell}\}_{\ell=1}^{\mu_k}$ generating $V(\omega_1+\omega_4)$ is given in Table~\ref{table:structure_of_rep_table}, where $\mu_k$ denotes the number of $k$-multi-indices in the generating set. In what follows we describe explicitly the invariant spinors, using a more sophisticated version of the technique from the proof of Proposition~\ref{HPn_twisted_inv_spinor_basis}.

		\begin{Theorem}\label{OP2_inv_spinors_theorem}
			A basis for the space of $F_4$-invariant $3$-twisted spin$^9$ spinors on $\OP^2$ is given~by
			\begin{align}\label{OP2_spinors_formula}
				\psi_{p}:= \sum_{k=0}^{18} \sum_{\ell=1}^{\mu_k} \bigl(\widehat{Y_{\mathcal{I}_{k,\ell}}.w_0}\bigr)^{\sharp}\otimes (Y_{\mathcal{I}_{k,\ell}}.w_p ) , \qquad p=1,2,3,4,
			\end{align}
			where $\sharp\colon \Sigma_{16}^* \to \Sigma_{16}$ is the musical isomorphism, $w_p$ $(p=1,2,3,4)$ denote highest weight vectors for the four copies of $\Sigma_{16}^-$ inside \smash{$\Sigma_9^{\otimes 3}$}, the indices $\mathcal{I}_{k,\ell}$ are as in Table~{\rm\ref{table:structure_of_rep_table}}, and for any $(Y_{\mathcal{I}_{k,\ell}}.w_0 )\in \Sigma_{16}^- \subseteq \Sigma_{16}$ we denote by \smash{$\widehat{Y_{\mathcal{I}_{k,\ell}}.w_0 }\in \Sigma_{16}^*$} the corresponding dual map sending $Y_{\mathcal{I}_{k',\ell'}}.w_0\mapsto \delta_{k,k'}\delta_{\ell,\ell'}$ and $\Sigma_{16}^+\mapsto 0$.
		\end{Theorem}
		\begin{proof}
			From the preceding discussion and Table~\ref{table:structure_of_rep_table}, we have the highest weight vector $w_0$ for~${\Sigma_{16}^- \subseteq \Sigma_{16}}$, together with explicit sequences of lowering operators $Y_i$ generating this subrepresentation. Altogether this gives four $\mathfrak{spin}(9)^{\C}$-module isomorphisms $T_{p}\colon \Sigma_{16}^- \to \Sigma_9^{\otimes 3}$, $p=1,2,3,4$ defined by
			\[
			T_{p} \colon\ Y_{\mathcal{I}_{k,\ell}}.w_0 \mapsto Y_{\mathcal{I}_{k,\ell}}.w_p, \qquad k=0,\dots, 18, \quad \ell = 1,\dots, \mu_k,
			\]
			where the $\mathcal{I}_{k,\ell}$ are as in Table~\ref{table:structure_of_rep_table} and we use the convention $Y_{\varnothing}= \Id$. By abuse of notation, we also denote by $T_p$ the extensions to all of $\Sigma_{16}$ by $\Sigma_{16}^+\mapsto 0$. The spinors $\psi_p$ are the images of the $T_p$ under the $\mathfrak{spin}(9)^{\C}$-module isomorphism $(\Sigma_{16}^-)^* \otimes \Sigma_9^{\otimes 3} \simeq \Sigma_{16}^- \otimes \Sigma_9^{\otimes 3}$, which are precisely given by~\eqref{OP2_spinors_formula}.
		\end{proof}
		
\begin{table}[th!]\renewcommand{\arraystretch}{1.23}
			\centering {\footnotesize
				\begin{tabular}{ |l|l|l| }
					\hline
					$k$ & $\mu_k$ & $\mathcal{I}_{k,\ell}\ (\ell=1,\dots, \mu_k )$ \\
					\hline
					\hline
					$0$ & $1$ & $\varnothing$ \\
					$1$ & $2$ & $\{1\}, \{4\}$ \\
					$2$ & $3$ & $\{2,1\},\{1,4\},\{3,4\}$ \\
					$3$ & $5$ & $\{3,2,1\},\{2,1,4\},\{1,3,4\},\{2,3,4\},\{4,3,4\}$ \\
					$4$ & $6$ & $\{4,3,2,1\},\{3,2,1,4\},\{2,1,3,4\},\{4,1,3,4\},\{1,2,3,4\},\{4,2,3,4\}$ \\
					$5$ & $8$ & $\{3,4,3,2,1\},\{4,4,3,2,1\},\{2,3,2,1,4\},\{4,3,2,1,4\},\{1,2,1,3,4\},$ \\ & & $\{4,2,1,3,4\},\{4,1,2,3,4\},\{3,4,2,3,4\}$ \\
					$6$ & $10$ & $\{2,3,4,3,2,1\},\{4,3,4,3,2,1\},\{3,4,4,3,2,1\},\{4,4,4,3,2,1\},\{1,2,3,2,1,4\},$\\ & & $\{4,2,3,2,1,4\},\{4,1,2,1,3,4\},\{3,4,2,1,3,4\},\{3,4,1,2,3,4\},\{4,3,4,2,3,4\}$ \\
					$7$ & $11$ & $\{1,2,3,4,3,2,1\},\{4,2,3,4,3,2,1\},\{4,4,3,4,3,2,1\},\{2,3,4,4,3,2,1\},$ \\ & & $ \{4,3,4,4,3,2,1\},\{4,1,2,3,2,1,4\},\{3,4,2,3,2,1,4\},\{3,4,1,2,1,3,4\},$ \\ & & $\{4,3,4,2,1,3,4\},\{2,3,4,1,2,3,4\},\{4,3,4,1,2,3,4\}$ \\
					$8$ & $12$ & $\{4,1,2,3,4,3,2,1\},\{3,4,2,3,4,3,2,1\},\{4,4,2,3,4,3,2,1\},\{3,4,4,3,4,3,2,1\},$ \\ & & $ \{4,4,4,3,4,3,2,1\},\{1,2,3,4,4,3,2,1\},\{4,2,3,4,4,3,2,1\},\{3,4,1,2,3,2,1,4\},$ \\ & & $\{4,3,4,2,3,2,1,4\},\{2,3,4,1,2,1,3,4\},\{4,3,4,1,2,1,3,4\},\{4,2,3,4,1,2,3,4\}$ \\
					$9$ & $12$ & $\{3,4,1,2,3,4,3,2,1\},\{4,4,1,2,3,4,3,2,1\},\{4,3,4,2,3,4,3,2,1\},\{3,4,4,2,3,4,3,2,1\},$ \\ & & $\{4,4,4,2,3,4,3,2,1\},\{2,3,4,4,3,4,3,2,1\},\{4,3,4,4,3,4,3,2,1\},\{4,1,2,3,4,4,3,2,1\},$ \\ & & $\{2,3,4,1,2,3,2,1,4\},\{4,3,4,1,2,3,2,1,4\},\{4,2,3,4,1,2,1,3,4\},\{3,4,2,3,4,1,2,3,4\}$ \\
					$10$ & $12$ & $\{2,3,4,1,2,3,4,3,2,1\},\{4,3,4,1,2,3,4,3,2,1\},\{3,4,4,1,2,3,4,3,2,1\},$ \\ & & $\{4,4,4,1,2,3,4,3,2,1\},\{4,4,3,4,2,3,4,3,2,1\},\{2,3,4,4,2,3,4,3,2,1\},$ \\ & & $\{4,3,4,4,2,3,4,3,2,1\},\{1,2,3,4,4,3,4,3,2,1\},\{4,2,3,4,4,3,4,3,2,1\},$ \\ & & $\{4,2,3,4,1,2,3,2,1,4\},\{3,4,2,3,4,1,2,1,3,4\},\{4,3,4,2,3,4,1,2,3,4\}$ \\
					$11$ & $11$ & $\{4,2,3,4,1,2,3,4,3,2,1\},\{4,4,3,4,1,2,3,4,3,2,1\},\{2,3,4,4,1,2,3,4,3,2,1\},$ \\ & & $\{4,3,4,4,1,2,3,4,3,2,1\},\{3,4,4,3,4,2,3,4,3,2,1\},\{4,4,4,3,4,2,3,4,3,2,1\},$ \\ & & $\{1,2,3,4,4,2,3,4,3,2,1\},\{4,2,3,4,4,2,3,4,3,2,1\},\{4,1,2,3,4,4,3,4,3,2,1\},$ \\ & & $\{3,4,2,3,4,1,2,3,2,1,4\},\{4,3,4,2,3,4,1,2,1,3,4\}$ \\
					$12$ & $10$ & $\{3,4,2,3,4,1,2,3,4,3,2,1\},\{4,4,2,3,4,1,2,3,4,3,2,1\},\{3,4,4,3,4,1,2,3,4,3,2,1\},$ \\ & & $\{4,4,4,3,4,1,2,3,4,3,2,1\},\{1,2,3,4,4,1,2,3,4,3,2,1\},\{4,2,3,4,4,1,2,3,4,3,2,1\},$ \\ & & $\{2,3,4,4,3,4,2,3,4,3,2,1\},\{4,3,4,4,3,4,2,3,4,3,2,1\},\{4,1,2,3,4,4,2,3,4,3,2,1\},$\\ & & $\{4,3,4,2,3,4,1,2,3,2,1,4\}$ \\
					$13$ & $8$ & $\{4,3,4,2,3,4,1,2,3,4,3,2,1\},\{3,4,4,2,3,4,1,2,3,4,3,2,1\},\{4,4,4,2,3,4,1,2,3,4,3,2,1\},$ \\ & & $\{2,3,4,4,3,4,1,2,3,4,3,2,1\},\{4,3,4,4,3,4,1,2,3,4,3,2,1\},\{4,1,2,3,4,4,1,2,3,4,3,2,1\},$ \\ & & $\{1,2,3,4,4,3,4,2,3,4,3,2,1\},\{4,2,3,4,4,3,4,2,3,4,3,2,1\}$ \\
					$14$ & $6$ & $\{4,4,3,4,2,3,4,1,2,3,4,3,2,1\},\{2,3,4,4,2,3,4,1,2,3,4,3,2,1\},$ \\ & & $\{4,3,4,4,2,3,4,1,2,3,4,3,2,1\},\{1,2,3,4,4,3,4,1,2,3,4,3,2,1\},$ \\ & & $\{4,2,3,4,4,3,4,1,2,3,4,3,2,1\},\{4,1,2,3,4,4,3,4,2,3,4,3,2,1\}$ \\
					$15$ & $5$ & $\{3,4,4,3,4,2,3,4,1,2,3,4,3,2,1\},\{4,4,4,3,4,2,3,4,1,2,3,4,3,2,1\},$ \\ & & $\{1,2,3,4,4,2,3,4,1,2,3,4,3,2,1\},\{4,2,3,4,4,2,3,4,1,2,3,4,3,2,1\},$ \\ & & $\{4,1,2,3,4,4,3,4,1,2,3,4,3,2,1\}$ \\
					$16$ & $3$ & $\{2,3,4,4,3,4,2,3,4,1,2,3,4,3,2,1\},\{4,3,4,4,3,4,2,3,4,1,2,3,4,3,2,1\},$\\ & & $\{4,1,2,3,4,4,2,3,4,1,2,3,4,3,2,1\}$ \\
					$17$ & $2$ & $\{1,2,3,4,4,3,4,2,3,4,1,2,3,4,3,2,1\},\{4,2,3,4,4,3,4,2,3,4,1,2,3,4,3,2,1\}$ \\
					$18$ & $1$ & $\{4,1,2,3,4,4,3,4,2,3,4,1,2,3,4,3,2,1\}$ \\		
					\hline
			\end{tabular} }
			\caption{Ordered sequences of lowering operators generating $V(\omega_1+\omega_4)$.}
			\label{table:structure_of_rep_table}
		\end{table}
		
		Finally, we give the differential equation satisfied by the spinors from Theorem~\ref{OP2_inv_spinors_theorem}. To begin, we need to first specify a connection on the vector bundle $\mathcal{A}$ associated to the auxiliary $\SO(9)$-bundle of the spin$^9$ structure. Note that $\mathcal{A}$ is associated to the principal $\Spin(9)$ bundle~${F_4\to F_4/\Spin(9)}$ by the composition of the covering map $\lambda_9\colon \Spin(9)\to \SO(9)$ with the standard representation $\rho_{\std}\colon \SO(9) \to \GL\bigl(\R^9\bigr)$. Indeed, there is a natural choice of invariant connection defined on $\mathcal{A}$ as follows. The structure of $\mathfrak{m}\simeq \Sigma_9^{\R} $ as a Clifford module for~$\Cl_9$ gives $9$ linearly independent endomorphisms, corresponding to Clifford multiplication by an orthonormal set of basis vectors for $\R^9$. By slightly modifying the Clifford multiplication (see~\cite[Section~2]{Fri01_weak_spin9}), one obtains endomorphisms $T_i\colon \mathfrak{m}\to\mathfrak{m}$, $i=1,\dots, 9$ satisfying the modified Clifford relations
		$
		T_i \circ T_j + T_j\circ T_i = 2\delta_{i,j}\Id$, $ T_i^* = T_i$, $ \tr T_i =0$
		for $i=1,\dots,9$. In this description, the isotropy image $\sigma (\Spin(9)) \subseteq \SO(\mathfrak{m}) \subseteq \text{End}(\mathfrak{m}) $ coincides with the normaliser of the $9$-dimensional subspace $\mathcal{T}:= \Span_{\R}\{T_1,\dots, T_9\} \subseteq \text{End}(\mathfrak{m}) $ \cite[p.~132]{Fri01_weak_spin9}:
		\[%\label{OP2_aux_bundle}
			\sigma(\Spin(9)) = \bigl\{ g\in \SO(\mathfrak{m})\mid g\mathcal{T}g^{-1} = \mathcal{T} \bigr\}.
		\]
		%By Lemma~\ref{lemma:reps}
		The $9$-dimensional $\Spin(9)$-representation $\mathcal{T}$ (action via conjugation) is isomorphic to $\rho_{\std}\circ \lambda_9$ (since there is only one irreducible real $9$-dimensional representation of $\Spin(9)$ up to isomorphism), hence we have
		\[
		\mathcal{A} \cong F_4 \times_{\Spin(9)} \mathcal{T} \subseteq F_4 \times_{\Spin(9)} \End(\mathfrak{m}) \cong \End\bigl( T\bigl(\mathbb{OP}^2\bigr)\bigr).
		\]
		In particular, $\mathcal{A}$ naturally inherits a connection $\nabla^{\End}$ from the extension of the Levi-Civita connection to the endomorphism bundle, whose Nomizu map vanishes on $\mathfrak{m}$. In light of Proposition~\ref{symmetricspace_differential_equation}, we finally see that the invariant twisted spinors found above are parallel:
		
		\begin{Theorem}
			\!The $4$-dimensional space of $F_4$-invariant $3$-twisted spin$^9$ spinors on $\mathbb{OP}^2$ is~span\-ned by parallel spinors for the connection $\nabla^{g,\End}:= \nabla^g \otimes \bigl(\nabla^{\End}\bigr)^{\otimes 3}$.
		\end{Theorem}
		
\subsection*{Acknowledgements}
		
The authors are grateful to Travis Schedler for his contributions to the representation-theoretical aspect of the paper, and to Marie-Am\'elie Lawn for her comments and fruitful discussions. We are grateful to the referees for their helpful comments.
D.~Artacho is funded by the UK Engineering and Physical Sciences Research Council (EPSRC), grant EP/W5238721. J.~Hofmann was supported by the Engineering and Physical Sciences Research Council [EP/L015234/1, EP/W522429/1]; the EPSRC Centre for Doctoral Training in Geometry and Number Theory (The London School of Geometry and Number Theory: University College London, King's College London, and Imperial College London); and a DAAD Short Term Research Grant for a~research stay at Philipps-Universit\"{a}t Marburg.
		
\pdfbookmark[1]{References}{ref}
\LastPageEnding
		
\end{document}